\documentclass[11pt,reqno]{amsart}
\usepackage{graphicx,epsfig,subfigure}
\newtheorem{thm}{Theorem}[section]
\newtheorem*{thm*}{Theorem}
\newtheorem{lem}[thm]{Lemma}
\newtheorem*{lem*}{Lemma}

\newtheorem{prop}[thm]{Proposition}

\theoremstyle{definition}

\renewcommand{\thecase}{}
\newtheorem*{case*}{Case}

\newtheorem*{defn*}{Definition}

\newtheorem*{exmp*}{Example}

\newtheorem{rmk}[thm]{Remark}
\newtheorem*{rmk*}{Remark}
\newtheorem{step}{Step}\renewcommand{\thestep}{}

\theoremstyle{remark}

\makeatletter
\def\alphenumi{
  \def\theenumi{\alph{enumi}}
  \def\p@enumi{\theenumi}
  \def\labelenumi{(\@alph\c@enumi)}}
\makeatother

\makeatletter
\def\thecase{\@arabic\c@case}
\makeatother

\makeatletter
\def\thestep{\@arabic\c@step}
\makeatother

\newcount\hh
\newcount\mm
\mm=\time
\hh=\time
\divide\hh by 60
\divide\mm by 60
\multiply\mm by 60
\mm=-\mm
\advance\mm by \time
\def\hhmm{\number\hh:\ifnum\mm<10{}0\fi\number\mm}

\let\oldmarginpar\marginpar
\renewcommand\marginpar[1]{\-\oldmarginpar[\raggedleft\footnotesize #1]%
{\raggedright\footnotesize #1}}

\newcommand\dotprod{\hbox{$\cdot$}}

\newcommand\CC{\mathbb{C}}

\newcommand\NN{\mathbb{N}}

\newcommand\QQ{\mathbb{Q}}
\newcommand\RR{\mathbb{R}}

\newcommand\cC{{\mathcal{C}}}

\newcommand\cF{{\mathcal{F}}}

\newcommand\cH{{\mathcal{H}}}
\newcommand\cI{{\mathcal{I}}}

\newcommand\cS{{\mathcal{S}}}

\newcommand\loc{\operatorname{loc}}

\newcommand\supp{\operatorname{supp}}

\newcommand\pr{\operatorname{pr}}

\newcommand\tb{{\tilde b}}

\newcommand\ty{{\tilde y}}

\newcommand\tbeta{{\tilde{\beta}}}

\numberwithin{equation}{section}

\usepackage{amssymb}
\usepackage{hyperref}
\usepackage{mathrsfs}
\usepackage{slashed}
\usepackage[usenames]{color}
\usepackage{url}
\usepackage{verbatim}
\newcommand{\hL}{\widehat{L}}
\newcommand{\hu}{\widehat{u}}

\hypersetup{pdftitle={Fractional Laplacian with drift}}

\hypersetup{pdfauthor={Charles L. Epstein and Camelia A. Pop}}

\begin{document}

\title[Fractional Laplacian with drift]{Regularity for the supercritical fractional Laplacian with drift}

\author[C. Epstein]{Charles L. Epstein}
\address[CE]{Department of Mathematics, University of Pennsylvania, 209 South
  33rd Street, Philadelphia, PA 19104-6395. Research of
    C.L.~Epstein partially supported by NSF grant
    DMS12-05851 and ARO grant W911NF-12-1-0552.}
\email{cle@math.upenn.edu}

\author[C. Pop]{Camelia A. Pop}
\address[CP]{Department of Mathematics, University of Pennsylvania, 209 South 33rd Street, Philadelphia, PA 19104-6395}
\email{cpop@math.upenn.edu}

\begin{abstract}
  We consider the linear stationary equation defined by the fractional
  Laplacian with drift. In the supercritical case, that is the case when the
  dominant term is given by the drift instead of the diffusion component, we
  prove local regularity of solutions in Sobolev spaces employing tools from
  the theory of pseudo-differential operators. The regularity of solutions in
  the supercritical case is as expected in the subcritical case, when the
  diffusion is at least as strong as the drift component, and the operator
  defined by the fractional Laplacian with drift can be viewed as an elliptic
  operator, which is not the case in the supercritical regime. We compute the
  leading part of the singularity of the Green's kernel, which display some
  rather unusual behavior when $0<s<\frac 12.$
\end{abstract}

%

\subjclass[2010]{Primary 35H99; secondary 60G22}
\keywords{Fractional Laplacian, pseudo-differential operators, Sobolev spaces, jump diffusion processes, symmetric stable processes, Markov processes}


\maketitle

\tableofcontents

\section{Introduction}
\label{sec:Intro}
We consider the linear operator defined by the fractional Laplacian with drift,
\begin{equation*}
Au(x) : = \left(-\Delta\right)^s u(x) + \nu(x)\dotprod\nabla u (x),
\end{equation*}
for $u \in \cS(\RR^n)$,  the Schwartz space
\cite[Definition (3.3.3)]{Taylor_vol1} consisting of smooth functions whose
derivatives of all order decrease faster than any polynomial at infinity. Here
$\nu:\RR^n\rightarrow\RR^n$ is a smooth, tempered vector field. The action of
the fractional Laplace operator $(-\Delta)^s$ on the space of Schwartz functions is defined through its Fourier representation by
\begin{equation*}
\widehat{(-\Delta)^s u}(\xi) = |\xi|^{2s}\widehat u(\xi).
\end{equation*}
The range of the parameter $s$ of particular interest is the interval
$(0,1)$. In our article, we study the case when $s\in (0,1/2)$. Recall that
$\cS'(\RR^n)$ is the dual space of $\cS(\RR^n),$ the space of tempered distributions.

The fractional Laplacian plays the same role in the theory of non-local
operators that the Laplacian plays in the theory of local elliptic
operators. For this reason, the regularity of solutions to equations defined by
the fractional Laplacian and its gradient perturbation is intensely studied in
the literature. It has applications to the study of a number of nonlinear
equations, as for example the quasi-geostrophic equation
\cite{Caffarelli_Vasseur_2010, Constantin_Wu_1999, Constantin_Wu_2008}, Burgers
equation \cite{Chan_Czubak_Silvestre_2010, Kiselev_2011}, and to the study of
the regularity of solutions and of the free boundary in the obstacle problem
defined by the fractional Laplacian with drift \cite{Petrosyan_Pop}. As the
infinitesimal generator of a Markov process with jumps described by a
$\alpha$-stable L\'evy process, the properties of its fundamental solution and
Green's function are studied in articles such as \cite{Bogdan_Jakubowski_2007,
  Chen_Kim_Song_2012, Jakubowski_2011, Jakubowski_Szczypkowsi_2010}, among
others.

The operator $A$ can be viewed as a pseudo-differential operator using the representation
\begin{equation}
\label{eq:Operator_pseudo}
Au(x) = (2\pi)^{-n} \int_{\RR^n} e^{ix\xi} \left(|\xi|^{2s} + i \nu(x)\dotprod\xi\right) \widehat u (\xi) \ d\xi,\quad \forall u \in \cS(\RR^n),
\end{equation}
with associated symbol,
\begin{equation}
\label{eq:Symbol_L}
a(x,\xi) := |\xi|^{2s} + i \nu(x)\dotprod\xi, \quad \forall x, \xi \in \RR^n.
\end{equation}
Strictly speaking $A$ is not a pseudo-differential operator because its symbol,
$a(x,\xi)$, is not smooth at $\xi=0$. For this reason $A$ does not map
$\cS(\RR^n)$ to itself. To facilitate our analysis we choose a smooth cut-off
function, $\varphi:\RR^n\rightarrow [0,1]$, such that
\begin{equation}
\label{eq:Cutoff_varphi}
\varphi(\xi)=0\quad\hbox{for } |\xi|<1,\quad\hbox{and}\quad \varphi(\xi)=1\quad\hbox{for } |\xi|>2,
\end{equation}
and we consider the symbol defined by
\begin{equation}
\label{eq:Cut_symbol}
\tilde a(x,\xi):=|\xi|^{2s}\varphi(2\xi)+i\nu(x)\dotprod\xi,\quad\forall x, \xi \in \RR^n.
\end{equation}
Let $\delta,\rho \in [0,1]$ and $m \in \RR$. We recall the definition of H\"ormander class of symbols, $S^m_{\rho,\delta}(\RR^n\times\RR^n)$, \cite[Definition (7.1.4)]{Taylor_vol2} which consists of smooth functions $p(x,\xi)$ such that, for all multi-indices $\alpha, \beta \in\NN^n$, there is a positive constant, $C_{\alpha,\beta}$, such that the following hold
$$
|D^{\alpha}_xD^{\beta}_{\xi}p(x,\xi)| \leq C_{\alpha,\beta} (1+|\xi|)^{m-\rho|\beta|+\delta|\alpha|},\quad\forall x,\xi\in\RR^n.
$$
To any symbol $p\in S^m_{\rho,\delta}(\RR^n\times\RR^n)$, we can associate the operator
$$
T_p u(x) := (2\pi)^{-n} \int_{\RR^n} e^{ix\xi} p(x,\xi) \widehat u (\xi) \ d\xi,\quad \forall u \in \cS(\RR^n),
$$
and the class of such operators is denoted by $OPS^m_{\rho,\delta}(\RR^n)$. 

In the language of pseudo-differential operators, we have that the symbol $\tilde a$ defined by \eqref{eq:Cut_symbol} belongs to the H\"ormander class $S^m_{1,0}(\RR^n\times\RR^n)$, where the order of the symbol is $m=1 \vee 2s$, and 
\begin{equation}
\label{eq:Decomposition}
Au = T_{\tilde a} u + Eu, \quad \forall u \in \cS(\RR^n),
\end{equation}
where we define the operator $E$ by an inverse Fourier transform,
\begin{equation}
\label{eq:Definition_E}
Eu(x):=\cF^{-1}\left[ |\xi|^{2s} (1-\varphi(2\xi)) \widehat u (\xi)\right],\quad \forall u \in \cS(\RR^n).
\end{equation}
Because the function $\xi\mapsto|\xi|^{2s}$ is not smooth at $\xi=0$, we see that $|\xi|^{2s}\widehat u$ is not necessarily a tempered distribution, when $u \in \cS'(\RR^n)$. We introduce the space
$$
\cS'_{(0,s)}(\RR^n):=\left\{u\in\cS'(\RR^n):|\xi|^{2s}\widehat u \in \cS'(\RR^n)\right\}.
$$
This is largely a condition on the behavior of the distribution at
``infinity.'' This space is the largest subspace of $\cS'(\RR^n)$ on which
$(-\Delta)^s$ makes sense.  We also let
\begin{equation}
  \hL^2_{\loc}(\RR^n)=\{u\in\cS'(\RR^n):\: \hu\in L^2_{\loc}(\RR^n)\}.
\end{equation}

If $u\in\cS'_{(0,s)}(\RR^n),$ then $|\xi|^{2s} (1-\varphi(2\xi)) \widehat u$ is a compactly supported distribution, and we have that $Eu(x)\in\cC^{\infty}_t(\RR^n),$ the space of smooth
functions all of whose derivatives have tempered growth, that is
$$
\cC^{\infty}_t(\RR^n):=\cC^{\infty}(\RR^n)\cap\cS'(\RR^n).
$$
From identity \eqref{eq:Decomposition}, and the fact that $T_{\tilde a}:\cS'(\RR^n)\rightarrow \cS'(\RR^n)$, we also have that
\begin{equation*}
A: \cS'_{(0,s)}(\RR^n) \rightarrow \cS'(\RR^n).
\end{equation*}
As $T_{\tilde a}$ is a classical pseudo-differential operator it also maps $\cS(\RR^n)$ to itself.

We recall that a tempered distribution, $u\in\cS'(\RR^n)$, belongs to the
Sobolev space, $H^m(\RR^n)$, where $m$ is a real number, if
$(1+|\xi|^2)^{m/2}\widehat u(\xi) \in L^2(\RR^n).$ We say that $u \in
H^m_{\loc}(\RR^n)$, if for any function  $\chi\in
\cC^{\infty}_c(\RR^n)$, we have that $\chi u$ belongs to $H^m(\RR^n)$. Using the facts that
$1-\varphi(2\xi)$ is a compactly supported, smooth function, and
$$
(1+|\xi|^2)^{m/2} \widehat{Eu}(\xi) = (1+|\xi|^2)^{m/2} |\xi|^{2s} (1-\varphi(2\xi)) \widehat u (\xi),\quad\forall m \in \RR,
$$
we obtain that if $u\in\hL^2_{\loc}(\RR^n),$ then $(1+|\xi|^2)^{m/2}
\widehat{Eu}$ belongs to $L^2(\RR^n).$ Thus the preceding argument shows that
\begin{align}
\label{eq:Definition_E_with_domain_distrib}
&E: \cS'_{(0,s)}(\RR^n) \rightarrow \cC^{\infty}_t(\RR^n)\\
\label{eq:Definition_E_with_domain_Sobolev}
&E: \hL^{2}_{\loc}(\RR^n) \rightarrow H^m(\RR^n),\quad\forall m \in \RR.
\end{align}

In this article, we establish local regularity in Sobolev spaces of solutions to the equation defined by the fractional Laplacian with drift, 
\begin{equation}
\label{eq:Equation}
Au(x)=f(x),\quad\forall x\in\RR^n,
\end{equation}
where the source function, $f$, is assumed to belong locally to the Sobolev
space $H^l(\RR^n)$, for some real constant $l$. Our strategy is to prove
existence of a two-sided parametrix for the pseudo-differential operator
$T_{\tilde a}$, and use it to establish the local regularity in Sobolev spaces
of solutions to equation \eqref{eq:Equation}. The construction of a two-sided
parametrix of the pseudo-differential operator $T_{\tilde a}$ is non-trivial in
the case when $s\in (0,1/2)$, the so-called supercritical regime, because the
symbol $\tilde a(x,\xi)$ is \emph{not} elliptic. When $s=1/2$ or $s\in
(1/2,1)$, the so-called critical and subcritical regime \cite[\S
1]{Silvestre_2012a, Caffarelli_Vasseur_2010}, respectively, $\tilde a(x,\xi)$ is
an elliptic symbol in the class $S^{2s}_{1,0}(\RR^n)$, and the existence of a
two-sided parametrix follows from well-known results \cite[\S
7.4]{Taylor_vol2}.  In this case, we assume that $u \in \cS'_{(0,s)}(\RR^n)$
is a solution to equation \eqref{eq:Equation} with source function $f\in
H^l_{\loc}(U)$, for $U$ an open set. The existence of a two-sided parametrix of
the operator $T_{\tilde a}$, together with identity \eqref{eq:Decomposition}
and \eqref{eq:Definition_E_with_domain_distrib}, show that $u$ also belongs to
$H^{l+2s}(V)$ for $V\subset\subset U.$ In the former case, when $s\in (0,1/2)$,
the drift dominates the diffusion component in the definition
\eqref{eq:Cut_symbol} of the symbol $\tilde a(x,\xi)$. The operator $\tilde a(x,D)$
is no longer elliptic and the existence of a parametrix for $T_{\tilde a}$ is
not obvious. Even so, by applying a change of coordinates dictated by the
vector field $\nu(x)$, we are able to build a two-sided parametrix for the
operator $T_{\tilde a}$, in the new system of coordinates. We use this property
to prove our main result. We let $\psi:\RR^n\rightarrow [0,1]$ be a smooth
cut-off function such that
\begin{equation}
\label{eq:Psi}
\psi(x)=1 \hbox{ for } |x|<1, \hbox{ and } \psi(x)=0 \hbox{ for } |x|>2,
\end{equation}
and we let 
\begin{equation}
\label{eq:Psi_r}
\psi_r(x)=\psi(x/r),\quad \forall r>0.
\end{equation}

We can now state
\begin{thm}[Local regularity of solutions]
\label{thm:Local_regularity}
Let $s\in (0,1/2)$, and let $x_0 \in \RR^n$. Assume that $\nu:\RR^n\rightarrow
\RR^n$ is a smooth vector field in a neighborhood of $x_0$, and that $\nu(x_0)
\neq 0$. Let  $l$ be a real constant, and assume that $u \in \cS'_{(0,s)}(\RR^n)$ is such
that $\psi_r(x-x_0) Au(x) \in H^l(\RR^n)$, for some positive constant $r$, where
$\psi_r$ is defined by \eqref{eq:Psi_r}. Then there is a positive constant,
$r_0<r$, such that for any smooth function, $\chi:\RR^n\rightarrow [0,1]$, with
compact support in $B_{r_0}(x_0)$, the function $\chi u \in H^{l+2s}(\RR^n)$
and $\chi\nu\dotprod\nabla u \in H^l(\RR^n)$.
\end{thm}

\subsection{Comparison with previous research}
\label{subsec:Previous_research}
Caffarelli and Vasseur \cite[Theorem 3]{Caffarelli_Vasseur_2010} study the
regularity of solutions to the \emph{evolution} equation associated to the
fractional Laplacian with drift in the critical case $s=1/2$. They assume that
the drift coefficient belongs to the BMO class of functions and that it is a
divergence free vector field to obtain regularity of solutions in H\"older
spaces. Silvestre extends the possible values of the parameter $s$, which now
can be chosen so that $s\in(0,1)$, and relaxes the assumptions on the drift
coefficient to obtain H\"older continuity of solutions in space and time in
\cite[Theorem 1.1]{Silvestre_2012a}. In particular, the drift coefficient is no
longer assumed to be divergence free. Instead, it is required to be bounded,
when $s\in [1/2,1)$, and to be H\"older continuous $\cC^{1-2s}$, when $s\in
(0,1/2)$. This result is improved in \cite[Theorem 1.1]{Silvestre_2012b}, where
the $\cC^{1,\alpha}$ H\"older continuity of solutions in space is proved, in the
case when $s\in (0,1/2]$. In this case, the drift coefficient is required to
belong to $\cC^{1-2s+\alpha}$, for some $\alpha \in (0,2s)$.

The difference between our work and that of Caffarelli-Vasseur
\cite{Caffarelli_Vasseur_2010} and Silvestre
\cite{Silvestre_2012a,Silvestre_2012b} consists in that we consider the
stationary, instead of the evolution equation defined by the fractional
Laplacian with drift, and we prove regularity of solutions in Sobolev spaces
instead of H\"older spaces. We assume that the vector field, $\nu$, is smooth
in order to be able to employ techniques from the theory of pseudo-differential
operators, while the work of Caffarelli-Vasseur \cite{Caffarelli_Vasseur_2010}
and Silvestre \cite{Silvestre_2012a,Silvestre_2012b} relaxes the regularity
assumptions on the drift coefficient and employs De Giorgi's approach to
parabolic equations and comparison arguments in order to obtain regularity of
solutions.

The fundamental solution of the fractional Laplacian with drift is studied in
\cite{Jakubowski_Szczypkowsi_2010}, in the case when $s\in (1/2,1)$. Assuming
that the vector field $\nu$ satisfies a certain integral condition
\cite[Inequality (4)]{Jakubowski_Szczypkowsi_2010}, the authors prove existence
\cite[Theorem 1]{Jakubowski_Szczypkowsi_2010}, and upper and lower bounds of a
fundamental solution comparable with the fundamental solution of the fractional
Laplacian without drift \cite[Theorems 2 and
3]{Jakubowski_Szczypkowsi_2010}. Similar results are proved in
\cite{Bogdan_Jakubowski_2007} and \cite{Jakubowski_2011} in the case when $s\in
(1/2,1)$, under different integral conditions on the vector field $\nu$. In
\cite{Bogdan_Jakubowski_2007}, the vector field in assumed to belong to a Kato
class of functions \cite[Definition 10]{Bogdan_Jakubowski_2007}, while in
\cite{Jakubowski_2011} the vector field is assumed to be divergence free, to
satisfy \cite[Condition (2)]{Jakubowski_2011}, and a ``smallness'' condition
\cite[Theorem 1]{Jakubowski_2011}. The Green function for the fractional
Laplacian with drift on smooth domains, in the case when $s\in (1/2,1)$, is
studied in \cite{Bogdan_Jakubowski_2012} and \cite{Chen_Kim_Song_2012}, under
the assumption that the vector field belongs to the Kato class of functions.
Regularity of solutions to the equation defined by the fractional Laplacian
with drift may be obtained using upper and lower bounds satisfied by the
fundamental solution. The preceding results establish existence and properties
of the fundamental solution when $s\in (1/2,1)$, while our article establishes
regularity of solutions when $s\in (0,1/2)$, using a method that circumvents
the use of pointwise estimates of the fundamental solution. In the last section
we determine the behavior of the leading singularity of the Green's function in
the supercritical case.

\subsection{Outline of the article}
\label{subsec:Outline}
In \S \ref{sec:First_construction_parametrix}, we build a two-sided parametrix
for $T_{\tilde a}$ in Lemma \ref{lem:Parametrix}, in the case when $s\in (1/4,
1/2)$. We use Lemma \ref{lem:Parametrix} to prove global regularity of solutions
to equation \eqref{eq:Equation}, Lemma
\ref{eq:Regularity_solutions_restrictive_case}. Our method is applicable only
in the case when $s\in (1/4, 1/2)$, because we use as a first approximation of
the parametrix a pseudo-differential operator with symbol in the H\"ormander
class $S^m_{\rho,\delta}(\RR^n\times\RR^n)$ with $\delta < \rho$, and this
property breaks down when $s\in (0, 1/4]$. In \S
\ref{sec:Second_construction_parametrix}, we take a different approach based on
a suitable change of coordinates to prove the existence of a two-sided
parametrix for $T_{\tilde a}$, for all $s\in (0, 1/2)$. In \S
\ref{subsec:Change_of_coordinates}, we describe the effect of the coordinate
change on the pseudo-differential operator $T_{\tilde a}$ in Lemma
\ref{lem:Change_of_coordinates_cut_symbol}. The second method for constructing
the parametrix and a localization procedure is then employed in \S
\ref{subsubsec:Localization} to prove the main result of our article, Theorem
\ref{thm:Local_regularity}.

In the last section we compute the leading order singularity of the Green's
Function, i.e.~the kernel $K(x,z),$ for $A^{-1},$ i.e.
\begin{equation}
  A^{-1}f(x)=\int\limits_{\RR^n}K(x,x-y)f(y)dy.
\end{equation}
In the supercritical case, this kernel
displays a interesting new feature: at $z\to 0$ it is more singular in the half
plane $\{y:y\cdot\nu(x)>0\}$ than in the complementary half plane. This is
an echo of the fact that the ``Green's function,'' for $(-\Delta)^s+\partial_{x_n},$
which is essentially a heat kernel, is supported in the half space
$\{(x',x_n):\: x_n>0\}.$ We have a short appendix where we analyze this kernel
as well.

\subsection{Notations and conventions}
\label{subsec:Notation}
We adopt the following definitions of the Fourier transform and the inverse Fourier transform of a function $u\in\cS(\RR^n)$,
\begin{align*}
\cF u(\xi) &= \widehat u (\xi) = \int_{\RR^n} e^{-ix\dotprod\xi} u(x)\ dx,\\
\cF^{-1} u(x) &=  (2\pi)^{-n} \int_{\RR^n} e^{ix\dotprod\xi} u(\xi)\ d\xi.
\end{align*}
These operators extend by duality to $\cS'(\RR^n).$ We let
\begin{equation}
  \langle x\rangle =(1+|x|^2)^{\frac 12}.
\end{equation}
Denote by $OPS^{-\infty}(\RR^n)$ the set of smoothing operators, that is
$$
OPS^{-\infty}(\RR^N) := \cap_{m=1}^{\infty} OPS^{-m}_{1,0}(\RR^n).
$$ 
For $n,m$ positive integers, we let $\cC^{\infty}(\RR^n;\RR^m)$ be the space of
$\RR^m$-valued smooth functions, and we let $\cC^{\infty}_b(\RR^n;\RR^m)$ be
the subspace of $\cC^{\infty}(\RR^n;\RR^m)$ of smooth $\RR^m$-valued functions
with bounded derivatives of all orders. The space $\cC^{\infty}_c(\RR^n;\RR^m)$
consists of smooth $\RR^m$-valued functions with compact support. For brevity, when $m=1$, we
write $\cC^{\infty}(\RR^n)$, $\cC^{\infty}_b(\RR^n)$ and
$\cC^{\infty}_c(\RR^n)$ instead of $\cC^{\infty}(\RR^n;\RR)$,
$\cC^{\infty}_b(\RR^n;\RR)$ and $\cC^{\infty}_c(\RR^n;\RR)$, respectively.

We denote by $\NN$ the extended set of natural numbers, that is $\NN:=\{0,1,2,\ldots\}$. Given real numbers, $a$ and $b$, we let $a \wedge b:=\min\{a,b\}$ and $a \vee b:=\max\{a,b\}$. For a $\cC^1$ function, $\phi:\RR^n\rightarrow\RR^n$, we let $J\phi(x)$ denote the Jacobian matrix, and by $|J\phi(x)|$ the Jacobian determinant, for all $x\in\RR^n$. For a matrix, $A \in \RR^{d\times d}$, we let $A^T$ denote the transpose matrix.

\section{A first construction of a two-sided parametrix}
\label{sec:First_construction_parametrix}
We build a two-sided parametrix for $T_{\tilde a}$ in the case when $s\in
(1/4,1/2)$. Our method cannot be extended to the case when $s\leq 1/4$, because
the candidate for a first approximation of the parametrix is a
pseudo-differential operator with symbol in the H\"ormander class
$S^m_{\rho,\delta}(\RR^n\times\RR^n)$ with $\delta \geq \rho$, and so the
calculus of such symbols does not allow us to build an asymptotic expansion for
the composition in terms of symbols of decreasing orders. In \S
\ref{sec:Second_construction_parametrix}, we take a different approach to prove
the existence of a two-sided parametrix for $T_{\tilde a}$, for all $s\in
(0,1/2)$.

\begin{lem}[A two-sided parametrix for $T_{\tilde a}$]
\label{lem:Parametrix}
Let $s\in (1/4,1/2)$, and $\nu\in \cC^{\infty}_b(\RR^n;\RR^n)$. Then the operator $T_{\tilde a}$ has a two-sided parametrix, $T_q$, where $q(x,\xi)$ is a symbol in the H\"ormander class $S^{-2s}_{2s,1-2s}(\RR^n \times \RR^n)$, that is
\begin{align}
\label{eq:Left_sided_parametrix}
T_q  T_{\tilde a} u &= u \hbox{ mod } \cC^{\infty}(\RR^n), \quad\forall u \in \cS'(\RR^n),\\
\label{eq:Right_sided_parametrix}
T_{\tilde a}  T_q u &= u \hbox{ mod } \cC^{\infty}(\RR^n), \quad\forall u \in \cS'(\RR^n).
\end{align}
\end{lem}

\begin{proof}
We build a left-parametrix, but the construction is the same for the right-parametrix. We let
$$
q_1(x,\xi):=\frac{\varphi(\xi)}{\tilde a(x,\xi)}, \quad \forall x, \xi \in \RR^n,
$$
where $\varphi:\RR^n \rightarrow [0,1]$ is a smooth cut-off function chosen as in \eqref{eq:Cutoff_varphi}. We recall that the symbol $\tilde a(x,\xi)$ is defined in \eqref{eq:Cut_symbol}. We consider a first approximation of the parametrix given by the pseudo-differential operator
$$
T_{q_1} u(x) = (2\pi)^{-n} \int_{\RR^n} e^{ix\xi} \frac{\varphi(\xi)}{|\xi|^{2s} + i \nu(x)\dotprod\xi} \widehat u(\xi) \ d\xi,\quad\forall u \in \cS(\RR^n),
$$
and we set $T_c= T_{q_1} \circ T_{\tilde a}$. By \cite[Proposition 7.3.3]{Taylor_vol2}, we expect the following asymptotic expansion to hold for $c$,
\begin{equation}
\label{eq:Asymptotic_expansion_c}
c(x,\xi) \sim \sum_{\alpha \geq 0} \frac{i^{|\alpha|}}{\alpha!} D^{\alpha}_{\xi} q_1(x,\xi) D^{\alpha}_x \tilde a(x,\xi),\quad\forall x,\xi\in\RR^n.
\end{equation}
To obtain the expansion \eqref{eq:Asymptotic_expansion_c}, we first look at the symbolic properties of $\tilde a(x,\xi)$ and $q_1(x,\xi)$. Because we assume that $s<1/2$, direct calculations give us that $D_{\xi_j} \tilde a(x,\xi)$ behaves like a symbol in $S^0_{1,0}(\RR^n \times \RR^n)$, and $D^{\alpha}_{\xi} \tilde a(x,\xi)$ behaves like a symbol in $S^{2s-|\alpha|}_{1,0}(\RR^n \times \RR^n)$, for all multi-indices $\alpha\in \NN^n$ such that $|\alpha|\geq 2$. In the case of the symbol $p_1(x,\xi)$, we obtain
\begin{align*}
D^{\alpha}_{x}q_1(x,\xi) &\sim (1+|\xi|)^{-2s+|\alpha|(1-2s)},\quad\forall \alpha \in \NN^n,\\
D^{\alpha}_{\xi} q_1(x,\xi) &\sim (1+|\xi|)^{-2s-|\alpha| 2s}, \quad\forall \alpha \in \NN^n.
\end{align*}
The preceding asymptotic properties of the symbol $q_1(x,\xi)$ show that $q_1(x,\xi)$ belongs to the H\"ormander class of symbols $S^{-2s}_{2s,1-2s}(\RR^n \times \RR^n)$. 

Therefore, \cite[Proposition 7.3.3]{Taylor_vol2} applies to our case and we obtain that the asymptotic expansion for $c(x,\xi)$ given by \eqref{eq:Asymptotic_expansion_c} holds. We denote, for all $j \in \NN$,
$$
c_j(x,\xi):=\sum_{\{\alpha\in\NN^n:|\alpha|=j\}} \frac{i^j}{\alpha!} D^{\alpha}_{\xi}q_1(x,\xi) D^{\alpha}_x \tilde a(x,\xi), \quad\forall x,\xi\in\RR^n.
$$
By construction, we have that $c_0(x,\xi)=1+(\varphi(\xi)-1)$, and the function $\varphi-1$ belongs to the class $S^{-\infty}(\RR^n \times \RR^n)$. Because the symbol $D_{\xi_j} q_1(x,\xi)$ belongs to $S^{-4s}_{2s,1-2s}(\RR^n \times \RR^n)$ and $D_{x_j}\tilde a(x,\xi)$ belongs to $S^1_{1,0}(\RR^n \times \RR^n)$, we obtain by \cite[Proposition 7.3.3]{Taylor_vol2} that $c_1(x,\xi)$ belongs to $S^{1-4s}_{2s,1-2s}(\RR^n \times \RR^n)$. Also, $c_j(x,\xi)$ belong to $S^{1-4s}_{2s,1-2s}(\RR^n \times \RR^n)$ (even better), for all $j \geq 1$, and so we have
$$
c(x,\xi) = 1+ r(x,\xi), \quad\forall x,\xi\in\RR^n,
$$
where $r \in S^{1-4s}_{2s,1-2s}(\RR^n \times \RR^n)$. Notice that by requiring that $s>1/4$, the order of the symbol $r$ is negative and also $\delta=1-2s<\rho=2s$.

We obtain that $T_{q_1}\circ T_{\tilde a}=T_{1+r}$. Because the symbol $r$ has negative order, the pseudo-differential operator $T_{1+r}$ is elliptic with symbol in the class $S^0_{2s,1-2s}(\RR^n \times \RR^n)$. By \cite[Section 7.4]{Taylor_vol2}, we can find a two-sided parametrix, $T_{q_2}$, with $q_2\in S^0_{2s,1-2s}(\RR^n \times \RR^n)$ such that $T_{q_2}\circ T_{1+r}=I+S_R$, where $S_R$ belongs to $\hbox{OPS}^{-\infty}(\RR^n)$. By setting $T_q=T_{q_2}\circ T_{q_1}$, we see that the symbol $q$ belongs to the class $S^{-2s}_{2s,1-2s}(\RR^n \times \RR^n)$, and the pseudo-differential operator $T_q$ is a left-parametrix for $T_{\tilde a}$.

Similarly, we can show that $T_{\tilde a}$ admits a right-parametrix, $T_p$, with symbol $p$ which belongs to the class $S^{-2s}_{2s,1-2s}(\RR^n \times \RR^n)$. We have, for all $u \in \cS'(\RR^n)$,
$$
T_p u+S_R\circ T_p u = \left(T_q\circ T_{\tilde a}\right)\circ T_p u = T_q\circ \left(T_{\tilde a}\circ T_p \right) u = T_q u + T_q\circ S_L u,
$$
and because $S_R\circ T_p$ and $T_q\circ S_L$ are pseudo-differential operators in $\hbox{OPS}^{-\infty}(\RR^n)$, we see that $T_q u=T_p u$ modulo $\cC^{\infty}(\RR^n)$ functions, for all tempered distributions $u \in \cS'(\RR^n)$ in a standard way.
\end{proof}

The construction of the left-parametrix in Lemma \ref{lem:Parametrix} can be used to prove regularity of solutions to the equation defined by the fractional Laplacian with drift \eqref{eq:Equation}.

\begin{lem}[Regularity of solutions]
\label{eq:Regularity_solutions_restrictive_case}
Let $s\in (1/4,1/2)$ and $k,l \in \RR$. Assume that the vector field $\nu\in \cC^{\infty}_b(\RR^n;\RR^n)$. Let $u \in H^k(\RR^n)$ be such that $Au \in H^{l}(\RR^n)$. Then $u \in H^{l+2s}(\RR^n)$ and $\nu\dotprod\nabla u \in H^l(\RR^n)$.
\end{lem}

\begin{proof}
Let $T_q$ be a left-parametrix of $T_{\tilde a}$ constructed as in the proof of Lemma \ref{lem:Parametrix}. Because we assume that $s>1/4$, we have that $2s>1-2s$, and using the fact that the symbol $q$ belongs to the class $S^{-2s}_{2s,1-2s}(\RR^n \times \RR^n)$, we obtain by \cite[Proposition 7.5.5]{Taylor_vol2} that the pseudo-differential operator
\begin{equation}
\label{eq:T_p_continuity}
T_q:H^l(\RR^n)\rightarrow H^{l+2s}(\RR^n)
\end{equation}
is bounded. 

Let $u \in H^k(\RR^n)$ be such that $Au\in H^l(\RR^n)$. We know
that $T_{\tilde a} u = A u -E u$ by \eqref{eq:Decomposition}.  By
\eqref{eq:Definition_E_with_domain_Sobolev} and the fact that $u \in H^k(\RR^n)\subset \hL^{2}_{\loc}(\RR^n)$, we have that $Eu \in H^m(\RR^n)$, for all
$m \in \RR$, and so $T_{\tilde a}u \in H^{l}(\RR^n)$. We know that
$T_q\circ T_{\tilde a} = I+R$, where the pseudo-differential operator $R\in OPS^{-\infty}(\RR^n)$, and so the map
$$
R:H^k(\RR^n)\rightarrow H^m(\RR^n),\quad\forall m \in \RR,
$$
is bounded. This fact together with the boundedness of the map in \eqref{eq:T_p_continuity} and the assumption $u\in H^k(\RR^n)$, gives us that $u \in
H^{l+2s}(\RR^n)$.  Since we have shown that $u\in H^{l+2s}(\RR^n)$ we
immediately conclude from the fact that $Au \in H^{l}(\RR^n)$ that 
$\nu\dotprod \nabla u\in H^l(\RR^n),$ as well. In fact, using symbolic
computations, one can show that $\nu\dotprod \nabla T_p:H^l(\RR^n)\rightarrow
H^{l}(\RR^n)$ is a bounded operator. This concludes the proof.
\end{proof}

\begin{rmk}[Hypotheses of Lemma \ref{eq:Regularity_solutions_restrictive_case}]
\label{rmk:Hyp_Reg_sol_restrictive_case}
The hypothesis that $u\in H^k(\RR^n)$, for some real constant $k$, in the statement of Lemma \ref{eq:Regularity_solutions_restrictive_case}, can be replaced with the assumption that $u\in \hL^2_{\loc}(\RR^n)$. In this case, we obtain that there is a constant $q\in\RR$ such that
\begin{equation}
\label{eq:u_in_weighted_Sobolev_space}
u \in H^{l+2s}(\RR^n)+\left\langle x\right\rangle^q H^{m}(\RR^n), \quad\forall m\in\RR^n,
\end{equation}
where $\left\langle x\right\rangle^q H^{m}(\RR^n)$ is the weighted Sobolev space consisting of tempered distributions of the form $\left\langle x\right\rangle^q u$ with $u\in H^m(\RR^n)$ \cite[Definition (2.150)]{Melrose_1998}. We can argue that \eqref{eq:u_in_weighted_Sobolev_space} holds in the following way. From \cite[Lemma 2.10]{Melrose_1998}, we have that there are real constants, $k$ and $q$, such that $u\in\left\langle x\right\rangle^q H^k(\RR^n)$, and from \cite[Theorem 2.4]{Melrose_1998} and the fact that $R\in\hbox{OPS}^{-\infty}(\RR^n)$, it follows that
$$
R:\left\langle x\right\rangle^q H^k(\RR^n)\rightarrow \left\langle x\right\rangle^q H^{k+m}(\RR^n),\quad\forall m\in\RR.
$$
Thus the identity $T_q\circ T_{\tilde a} = I+R$ now gives us \eqref{eq:u_in_weighted_Sobolev_space}.
\end{rmk}

\section{A second construction of a two-sided parametrix}
\label{sec:Second_construction_parametrix}
In this section, we extend the results of \S
\ref{sec:First_construction_parametrix} from the case when $s\in (1/4,1/2)$ to
that when $s\in (0,1/2)$. We give a method to build a two-sided parametrix for
the pseudo-differential operator $T_{\tilde a}$ which can be applied for all
parameters $s$ in the range $(0,1/2)$. From the preceding section, we see that
the reason we had to restrict to the case $s\in (1/4,1/2)$ is that the
candidate for a first approximation of the parametrix of $T_{\tilde a}$ is a
pseudo-differential operator with symbol, $q_1(x,\xi)$, that belongs to the
H\"ormander class $S^{-2s}_{2s,1-2s}(\RR^n \times \RR^n)$, and we need to have
$2s>1-2s$, that is $s>1/4$, in order to obtain a useful symbolic composition
formula.  We notice that if the vector field $\nu(x)$ is a constant, then the
symbol $q_1(x,\xi)$ belongs to the class $S^{-2s}_{2s,0}(\RR^n \times \RR^n)$,
and so we do not need to assume any restriction on the values of the parameter
$s$, in order to obtain a two-sided parametrix for the operator $T_{\tilde
  a}$. Our strategy for a general smooth vector field $\nu(x)$, is to apply a
change of coordinates to obtain an operator with a constant drift for which we
can construct a two-sided parametrix using the methods of \S
\ref{sec:First_construction_parametrix}.

\subsection{Change of coordinates}
\label{subsec:Change_of_coordinates}
We first describe the effect of the change of coordinates on the symbols of
pseudo-differential operators. While this is a classical subject, we were not
able to find the statement of Lemma \ref{lem:Change_of_coordinates} below in
the form required in this article. However, analogous statements which hold
for symbols with kernels with compact support can be found in H\"ormander
\cite[Theorem 18.1.17]{Hormander_vol3} (see also the comment on
\cite[p. 352]{Hormander_vol3} for symbols in the class
$S^m_{\rho,\delta}(\RR^n\times\RR^n)$), and for symbols in the class
$S^m_{1,0}(\RR^n\times\RR^n)$, can be found in \cite{Abels}.

For functions $\Phi:\RR^n\rightarrow\RR^n$ and $u:\RR^n\rightarrow\RR$, we define
$$
(\Phi^*u)(x):= u(\Phi(x)),\quad\forall x\in \RR^n.
$$

\begin{lem}[Change of coordinates]
\label{lem:Change_of_coordinates}
Let $a$ be a symbol in the H\"ormander class $S^m_{\rho,\delta}(\RR^n\times\RR^n)$, where $0\leq\delta\leq 1$ and $0\leq 1-\rho<\rho\leq 1$. Let $\Phi:\RR^n\rightarrow\RR^n$ be such that
\begin{enumerate}
\item \label{item:Diffeomorphism_invertibility_smoothness} $\Phi$ is a bijection, and the functions $\Phi$ and $\Phi^{-1}$ belong to $\cC^{\infty}_b(\RR^n;\RR^n)$.
\item \label{item:Auxiliary_function_invertibility_smoothness} The inverse of
  the matrix-valued function
  $H:\RR^n\times\RR^n\rightarrow\RR^{n^2},$
\begin{equation}
\label{eq:Auxiliary_function_invertibility_smoothness}
H(x,y) = \int_{0}^1 J\Phi(x+t(y-x)) \ dt,\quad\forall x,y \in\RR^n,
\end{equation}
is defined at all points $(x,y)\in\RR^n\times\RR^n$, and the function $(x,y)\to
H^{-1}(x,y)$ belongs to $\cC^{\infty}_b(\RR^n\times\RR^n;\RR^{n^2})$.
\end{enumerate}
Then there is a symbol, $b \in S^m_{\rho,\delta'}(\RR^n\times\RR^n)$ with $\delta'=\delta\vee(1-\rho)$, such that
\begin{equation}
\label{eq:Operator_change_of_coordinates}
\Phi^* T_a (\Phi^{-1})^* = T_b,
\end{equation}
Moreover, the symbol $b$ satisfies, for all non-negative integers $N$,
\begin{equation}
\label{eq:Operator_change_of_coordinates_asymptotic_expansion}
b(y,\eta) - \sum_{|\alpha|< N} \frac{1}{\alpha!} D^{\alpha}_{\eta} D^{\alpha}_w c(y,w,\eta)|_{w=y}
\in S^{m-N(2\rho-1)}_{\rho,\delta'}(\RR^n\times\RR^n),
\end{equation}
where
\begin{equation}
\label{eq:Compound_symbol_change_of_coordinates}
c(y,w,\eta)= a(\Phi(y),H(y,w)^{-T} \eta) |J\Phi(w)| |H(y,w)^{-T}|,\quad\forall y,w,\eta\in\RR^n.
\end{equation}
\end{lem}

\begin{rmk}[Invariance of classes of pseudo-differential operators under change of coordinates]
\label{rmk:Invariance_symbol_class}
Lemma \ref{lem:Change_of_coordinates} shows that given a symbol, $a\in S^m_{\rho,\delta}(\RR^n\times\RR^n)$, with $1-\rho<\rho$, there is a symbol $b \in S^m_{\rho,\delta'}(\RR^n\times\RR^n)$, with $\delta'=\delta\vee (1-\rho)$, such that the change of coordinates \eqref{eq:Operator_change_of_coordinates} holds. To guarantee invariance of symbol classes under change of coordinates, we need to require that 
\begin{equation}
\label{eq:Condition_invariance_symbols}
0\leq 1-\rho\leq \delta<\rho\leq 1,
\end{equation}
When the preceding condition is satisfied, we see that $\delta'=\delta$, and both symbols $a(x,\xi)$ and $b(x,\xi)$ belong to the same class. Condition \eqref{eq:Condition_invariance_symbols} is consistent with the hypothesis of \cite[Theorem 18.1.17]{Hormander_vol3} for symbols in the H\"ormander class $S^m_{1,0}(\RR^n\times\RR^n)$ with kernels with compact support (see also the comment on \cite[p. 352]{Hormander_vol3}). 
\end{rmk}

\begin{proof}[Proof of Lemma \ref{lem:Change_of_coordinates}]
Let $u \in \cS(\RR^n)$ and define $v=\Phi^* u$ and $x=\Phi(y)$. Then $u=(\Phi^{-1})^* v$, and the left-hand side of identity \eqref{eq:Operator_change_of_coordinates} becomes
\begin{align*}
\Phi^* T_a (\Phi^{-1})^* v(y) &= T_a (v\circ\Phi^{-1}) (\Phi(y)) (= T_a u(x))\\
&= (2\pi)^{-n} \int e^{i\Phi(y)\dotprod\xi} a(\Phi(y),\xi) \widehat{v\circ\Phi^{-1}}(\xi)\ d\xi\\
&= (2\pi)^{-n}\int e^{i(\Phi(y)-z)\xi} a(\Phi(y),\xi) v(\Phi^{-1}(z))\ dz d\xi \\
&= (2\pi)^{-n}\int e^{i(\Phi(y)-\Phi(w))\xi} a(\Phi(y),\xi) |J\Phi(w)| v(w)\ dw d\xi,
\end{align*}
where we applied the change of variables $z=\Phi(w)$ in the penultimate equality. Using identity \eqref{eq:Auxiliary_function_invertibility_smoothness} we have that
$$
\Phi(y)-\Phi(w)=H(y,w)(y-w),\quad\forall y, w\in\RR^n,
$$
and applying the change of variable $\eta=H(y,z)^T\xi$, we obtain
\begin{align*}
\Phi^* T_a (\Phi^{-1})^* v(y) 
&= (2\pi)^{-n}\int e^{i(y-w)\eta} a(\Phi(y),H(y,z)^{-T} \eta) |J\Phi(w)| |H(y,w)^{-T}| v(w)\ dw d\xi.
\end{align*}
Thus, the pseudo-differential operator $\Phi^* T_a (\Phi^{-1})^*$ is defined as
an operator with a compound symbol $c(y,w,\eta)$ given by
\eqref{eq:Compound_symbol_change_of_coordinates}. The compound symbol
$c(y,w,\eta)$ belongs to the class
$S^m_{\rho,\delta,1-\rho}(\RR^{2n}\times\RR^n)$ (see definition \cite[\S 7.3,
Inequality (3.3)]{Taylor_vol2}). This follows from the fact that the symbol $a
\in S^m_{\rho,\delta}(\RR^n\times\RR^n)$, the functions $\Phi$ and $\Phi^{-1}$
belong to $\cC^{\infty}_b(\RR^n;\RR^n)$ by assumption
\eqref{item:Diffeomorphism_invertibility_smoothness}, and the map $H(y,w)$
satisfies assumption \eqref{item:Auxiliary_function_invertibility_smoothness}
in the statement of the lemma. Because we also assume that $1-\rho<\rho$, it follows from \cite[Proposition 7.3.1]{Taylor_vol2} that there is a symbol $b \in S^m_{\rho,\delta\vee(1-\rho)}(\RR^n\times\RR^n)$ such that identity \eqref{eq:Operator_change_of_coordinates} holds, and the symbol $b$ satisfies the asymptotic expansion \eqref{eq:Operator_change_of_coordinates_asymptotic_expansion}.
\end{proof}

We apply Lemma \ref{lem:Change_of_coordinates} to the pseudo-differential operator $T_{\tilde a}$, defined by the symbol $\tilde a$ in \eqref{eq:Cut_symbol}.

\begin{lem}[Change of coordinates for $T_{\tilde a}$]
\label{lem:Change_of_coordinates_cut_symbol}
Let $s\in (0,1/2)$, and let $\Phi:\RR^n\rightarrow\RR^n$ be a diffeomorphism
which satisfies assumptions
\eqref{item:Diffeomorphism_invertibility_smoothness} and
\eqref{item:Auxiliary_function_invertibility_smoothness} of Lemma
\ref{lem:Change_of_coordinates}. Let $\nu\in \cC^{\infty}_b(\RR^n;\RR^n)$ be a
smooth vector field such that, for all functions $u\in \cC^1(\RR^n)$, we have
that
\begin{equation}
\label{eq:Derivatives_change_of_coordinates}
\nu(x)\dotprod \nabla u(x) = v_{y_n}(y),
\end{equation}
where we let $v:=\Phi^*u$ and $x=\Phi(y)$. Then there is a symbol, $b \in
S^{1}_{1,0}(\RR^n\times\RR^n)$, such that
\begin{equation}
\label{eq:Asymptotic_expansion_b}
b(y,\eta) - \left(|(J\Phi(y))^{-T}\eta|^{2s} \varphi(2(J\Phi(y))^{-T}\eta)+i\eta_n\right) \in S^{2s-1}_{1,0}(\RR^n\times\RR^n), 
\end{equation}
and 
\begin{equation}
\label{eq:Change_of_coordinates_b}
\Phi^* T_{\tilde a} (\Phi^{-1})^* = T_b.
\end{equation}
\end{lem}

\begin{proof} We write 
\begin{equation}
\label{eq:Split_tilde_a}
\tilde a(x,\xi)= \tilde a^1(\xi)+\tilde a^2(x,\xi),
\end{equation}
where we let
\begin{equation}
\label{eq:Tilde_a_1_2}
\tilde a^1(\xi)=|\xi|^{2s}\varphi(2\xi)\quad\hbox{and}\quad \tilde a^2(x,\xi)=i\nu(x)\xi.
\end{equation}
Clearly the symbol $\tilde a^1$ belongs to the class $S^{2s}_{1,0}(\RR^n\times\RR^n)$ and $\tilde a^2$ belongs to $S^1_{1,0}(\RR^n\times\RR^n)$, and we have that
$$
\Phi^* T_{\tilde a} (\Phi^{-1})^* = \Phi^* T_{\tilde a^1} (\Phi^{-1})^* + \Phi^* T_{\tilde a^2} (\Phi^{-1})^*.
$$
Identity \eqref{eq:Derivatives_change_of_coordinates} immediately gives us that $\Phi^* T_{\tilde a^2} (\Phi^{-1})^*=T_{b^2}$ with $b^2(y,\eta)=i\eta_n$. Applying \eqref{eq:Operator_change_of_coordinates_asymptotic_expansion} with $N=1$ to the symbol $\tilde a^1$, we obtain that
$$
\Phi^* T_{\tilde a^1} (\Phi^{-1})^*=T_{b^1},
$$
and the symbol $b^1$ satisfies
$$
b^1(y,\eta) - |(J\Phi(y))^{-T}\eta|^{2s} \varphi(2(J\Phi(y))^{-T}\eta) \in S^{2s-1}_{1,0}(\RR^n\times\RR^n).
$$
To obtain the preceding expression, we used the fact that $H(y,y)^{-1}=(J\Phi(y))^{-1}$, and that $|J\Phi(y)||H(y,y)^{-T}|=1$.
Letting now 
$$
b(y,\eta)=b^1(y,\eta)+b^2(y,\eta),
$$
we see that the symbol $b$ satisfies \eqref{eq:Asymptotic_expansion_b} and identity \eqref{eq:Change_of_coordinates_b}.
\end{proof}

\subsection{Construction of a two-sided parametrix and regularity of solutions}
\label{subsec:Parametrix_regularity}
We now build a two-sided parametrix for the pseudo-differential operator $T_b$
(Lemma \ref{lem:Parametrix_b}), which we then use to prove local regularity of solutions in Sobolev spaces to equation \eqref{eq:Equation} (Lemma \ref{lem:Regularity_solutions_general_case}). The hypotheses of Lemmas \ref{lem:Parametrix_b} and \ref{lem:Regularity_solutions_general_case} are relaxed in \S \ref{subsubsec:Diffeomorphism},  where in Lemma \ref{lem:Construction_diffeomorphism_local} we give sufficient conditions for the existence of a diffeomorphism, $\Phi$, with suitable local properties. Then, in \S \ref{subsubsec:Localization} we use Lemma \ref{lem:Construction_diffeomorphism_local} to prove the main result of our article, Theorem \ref{thm:Local_regularity}.

\begin{lem}[A two-sided parametrix for $T_b$]
\label{lem:Parametrix_b}
Let $s\in (0,1/2)$, and let $\Phi:\RR^n\rightarrow\RR^n$ be a diffeomorphism
which satisfies assumptions
\eqref{item:Diffeomorphism_invertibility_smoothness} and
\eqref{item:Auxiliary_function_invertibility_smoothness} of Lemma
\ref{lem:Change_of_coordinates}. Let $\nu\in \cC^{\infty}_b(\RR^n;\RR^n)$ be a
smooth vector field such that identity
\eqref{eq:Derivatives_change_of_coordinates} holds for all functions $u\in
\cC^1(\RR^n)$, where we let $v:=\Phi^*u$. In addition we assume that there is a
positive constant, $C$, such that
\begin{equation}
\label{eq:Diffeomorphism_nondegeneracy}
C^{-1}|\eta| \leq |(J\Phi(y))^{-T}\eta| \leq C|\eta|,\quad\forall y, \eta\in\RR^n.
\end{equation}
Then there is a symbol, $p \in S^{-2s}_{2s,0}(\RR^n\times\RR^n)$, such that
\begin{equation}
\label{eq:Parametrix_b}
T_p T_b = T_b T_p = I \hbox{ mod } OPS^{-\infty}(\RR^n).
\end{equation}
\end{lem}

\begin{proof}
From \eqref{eq:Asymptotic_expansion_b}, it follows that the leading part of $b$ is given by 
\begin{equation*}
b_0(y,\eta) =  |(J\Phi(y))^{-T}\eta|^{2s} \varphi(2(J\Phi(y))^{-T})\eta)+i\eta_n,\quad\forall y,\eta\in\RR^n,
\end{equation*}
and the symbol $b_0 \in S^{1}_{1,0}(\RR^n\times\RR^n)$. We use the symbol
$b(y,\eta)$ to build a left- and right-parametrix for the
pseudo-differential operator $T_b$ exactly as in the proof of Lemma
\ref{lem:Parametrix}. We first define
$$
p_1(y,\eta):=\frac{\psi(\eta)}{b_0(y,\eta)},
$$
where $\supp\psi\subset\{\eta:\:\varphi(2(J\Phi(y))^{-T})\eta)=1\}.$ 
Reviewing the proof of Lemma \ref{lem:Parametrix}, we see that $p_1$
defines a symbol in the H\"ormander class $S^{-2s}_{2s,0}(\RR^n\times\RR^n)$,
while previously we had that $p_1$ belongs to
$S^{-2s}_{2s,1-2s}(\RR^n\times\RR^n)$ due to the (non-zero) derivatives of the
vector field $\nu(x)$, which now are zero because of the condition
\eqref{eq:Derivatives_change_of_coordinates}. We notice that condition
\eqref{eq:Diffeomorphism_nondegeneracy} guarantees that the symbol $p_1$ indeed
belongs to the class $S^{-2s}_{2s,0}(\RR^n\times\RR^n)$. The remaining part of
the proof of Lemma \ref{lem:Parametrix} applies without the restriction
$2s>1-2s$, that is $s>1/4$, to the present setting, and we obtain a two-sided parametrix, $T_p$, with a
symbol $p(y,\eta)$ belonging to the class $S^{-2s}_{2s,0}(\RR^n\times\RR^n)$.
\end{proof}

We can now state the regularity result in the more general case when $0<s<1/2$.

\begin{lem}[Regularity of solutions]
\label{lem:Regularity_solutions_general_case}
Let $s\in (0,1/2)$, and let $\Phi:\RR^n\rightarrow\RR^n$ be a diffeomorphism which satisfies the hypotheses of Lemma \ref{lem:Parametrix_b}. Let $u \in H^k(\RR^n)$ be such that $Au \in H^l(\RR^n)$, for some real constants, $k$ and $l$. Then $u \in H^{l+2s}(\RR^n)$ and $\nu\dotprod\nabla u \in H^l(\RR^n)$.
\end{lem}

\begin{rmk}
The hypotheses of Lemma \ref{lem:Regularity_solutions_general_case} are relaxed in Theorem \ref{thm:Local_regularity}.
\end{rmk}

To prove Lemma \ref{lem:Regularity_solutions_general_case}, we use the following

\begin{prop}\cite[\S 4.2]{Taylor_vol1}
\label{prop:Sobolev_spaces_diffeomorphism}
Let $\Psi:\RR^n\rightarrow\RR^n$ be a diffeomorphism with bounded derivatives of all orders. Assume there is a positive constant, $C$, such that
$$
|J\Psi^{-1}(x)| \leq C,\quad\forall x \in \RR^n.
$$
Then, for all $l\in\RR$,
$$
\Psi^*:H^l(\RR^n)\rightarrow H^l(\RR^n)
$$
is a continuous linear map.
\end{prop}

\begin{proof}[Proof of Lemma \ref{lem:Regularity_solutions_general_case}] 
From Lemma \ref{lem:Parametrix_b}, we obtain that there is a symbol $p \in S^{-2s}_{2s,0}(\RR^n\times\RR^n)$ such that
$$
T_p T_b = I \hbox{ mod } OPS^{-\infty}(\RR^n).
$$
Thus, if $w\in H^k(\RR^n)$ is such that $T_b w \in H^l(\RR^n)$, then the
previous identity implies that $w$ belongs to $H^{l+2s}(\RR^n)$ and $w_{y_n}$
belongs to $H^{l}(\RR^n)$. We will use this fact together with Proposition
\ref{prop:Sobolev_spaces_diffeomorphism} to prove that $u \in H^{l+2s}(\RR^n)$
and $\nu\dotprod\nabla u \in H^{l}(\RR^n)$.

We first show that $T_b v$ belongs to the Sobolev space $H^l(\RR^n)$. Recall from identity \eqref{eq:Change_of_coordinates_b} that $T_b v(y)=T_{\tilde a} u (x)$, where we recall that $v=\Phi^* u$  and $x=\Phi(y)$, which implies that $T_b v(y) = Au(x) -Eu(x)$, where $Eu$ is defined in \eqref{eq:Definition_E}. Therefore, we have that
$$
T_b v(y)=\Phi^*(Au)(y) - \Phi^*(Eu)(y).
$$
By hypothesis, we have that $Au \in H^l(\RR^n)$, and so Proposition
\ref{prop:Sobolev_spaces_diffeomorphism} implies that $\Phi^*(Au)$ also belongs
to $H^l(\RR^n)$. From our assumption that $u \in H^k(\RR^n)\subset \hL_{\loc}(\RR^n)$ and
property \eqref{eq:Definition_E_with_domain_Sobolev}, we see that because $Eu$ belongs
to $H^l(\RR^n)$.  From Proposition \ref{prop:Sobolev_spaces_diffeomorphism}, it
follows that $\Phi^*(Eu) \in H^l(\RR^n)$, and so we can conclude that $T_b v$
belongs to $H^l(\RR^n)$ also. Thus the function $v$ belongs to
$H^{l+2s}(\RR^n)$ and $v_{y_n}$ belongs to $H^{l}(\RR^n)$. Recall that we have
$$
u=(\Phi^{-1})^*v\quad\hbox{and}\quad \nu\dotprod\nabla u =(\Phi^{-1})^*v_{y_n}.
$$
Applying Proposition \ref{prop:Sobolev_spaces_diffeomorphism} again implies
that the function $u$ is in $H^{l+2s}(\RR^n)$ and $\nu\dotprod\nabla u$ is in
$H^{l}(\RR^n)$.
\end{proof}

\begin{rmk}[Hypothesis of Lemma \ref{lem:Regularity_solutions_general_case}]
\label{rmk:Hyp_Reg_sol_all_s}
As in Remark \ref{rmk:Hyp_Reg_sol_restrictive_case}, the hypothesis that $u\in H^k(\RR^n)$, for some real constant $k$, in the statement of Lemma \ref{lem:Regularity_solutions_general_case}, can be replaced with the assumption that $u\in \hL^2_{\loc}(\RR^n)$, and we obtain that \eqref{eq:u_in_weighted_Sobolev_space} holds.
\end{rmk}

\begin{rmk}[Parametrix for $T_{\tilde a}$]
In Lemma \ref{lem:Parametrix_b}, we prove the existence of a two-sided parametrix for the pseudo-differential operator $T_b$, which is given by the operator $T_p$, with symbol $p$ in the H\"ormander class $S^{-2s}_{2s,0}(\RR^n\times\RR^n)$. Because the pseudo-differential operator $T_b$ satisfies identity \eqref{eq:Change_of_coordinates_b}, we can see that the operator
\begin{equation}
\label{eq:Parametrix_Q}
Q = (\Phi^{-1})^* T_p \Phi^*,
\end{equation}
is a two-sided parametrix for $T_{\tilde a}$, that is
$$
Q T_{\tilde a} u = T_{\tilde a} Q u = u \hbox{ mod } \cC^{\infty}(\RR^n),\quad\forall u \in \cS(\RR^n).
$$
We now discuss to what extent the two-sided parametrix of the operator $T_{\tilde a}$, $Q$, can be represented as a pseudo-differential operator.

When $s\in (1/4,1/2)$, we have that $1-2s<2s<1$. Therefore Lemma \ref{lem:Change_of_coordinates} shows that $Q$ can be represented as a pseudo-differential operator with symbol in the class $S^{-2s}_{2s,1-2s}(\RR^n\times\RR^n)$, which is consistent with Lemma \ref{lem:Parametrix}. Conversely, we can start from the two-sided parametrix, $T_q$, of the pseudo-differential operator $T_{\tilde a}$, given by Lemma \ref{lem:Parametrix}. Now $q$ is a symbol in the H\"ormander class $S^{-2s}_{2s,1-2s}(\RR^n\times\RR^n)$. According to the previous argument, 
$$
P = \Phi^* T_q (\Phi^{-1})^*,
$$
is a two-sided parametrix for $T_b$. Because the symbol $q$ belongs to $S^{-2s}_{2s,1-2s}(\RR^n\times\RR^n)$ and we assume that $1-2s<2s$, Lemma \ref{lem:Change_of_coordinates} gives us that $P$ can be represented as a pseudo-differential operator with symbol $p \in S^{-2s}_{2s,1-2s}(\RR^n\times\RR^n)$. In Lemma \ref{lem:Parametrix_b}, we show that $p$ is actually in a better symbol class, that is it belongs to $S^{-2s}_{2s,0}(\RR^n\times\RR^n)$.

When $s\in (0, 1/4]$, the condition $1-2s<2s$ is no longer fulfilled, and so we
cannot guarantee that we can represent $Q$ as a pseudo-differential
operator. This explains the difficulty in the construction of a two-sided
parametrix for the operator $T_{\tilde a}$ using the direct approach of \S
\ref{sec:First_construction_parametrix}.
\end{rmk}

\subsubsection{Construction of a diffeomorphism}
\label{subsubsec:Diffeomorphism}
We now want to build a diffeomorphism, $\Phi:\RR^n\rightarrow\RR^n$, which satisfies assumptions \eqref{item:Diffeomorphism_invertibility_smoothness} and \eqref{item:Auxiliary_function_invertibility_smoothness} of Lemma \ref{lem:Change_of_coordinates}, verifies \emph{locally} identity \eqref{eq:Derivatives_change_of_coordinates} of Lemma \ref{lem:Change_of_coordinates_cut_symbol}, and verifies condition \eqref{eq:Diffeomorphism_nondegeneracy} of Lemma \ref{lem:Parametrix_b}.

Let $\nu$ be a smooth vector field. Because we want identity \eqref{eq:Derivatives_change_of_coordinates} to be satisfied only \emph{locally}, we fix a point $x_0\in\RR^n$, and we assume that $\nu(x_0)\neq 0$. We may assume choose a system of coordinates such that $x_0=O$ and $\nu\dotprod e_n\neq 0$. We define the smooth vector field $\mu$ by
\begin{equation}
\label{eq:New_drift}
\mu(x) = \psi_r(x) \nu(x) + (1-\psi_r(x)) \nu(O),\quad\forall x \in \RR^n,
\end{equation}
where the positive constant $r$ will be suitably chosen below. We recall that the cut-off function $\psi_r$ is defined in \eqref{eq:Psi_r}. The vector field $\mu$ is a globally Lipschitz function and it is constant outside a compact set. By the Picard-Lindel\"of Theorem \cite[Theorem II.1.1]{Hartman}, for any $y'\in\RR^{n-1}$, there is a unique global solution to
\begin{equation}
\label{eq:ODE}
\begin{aligned}
\begin{cases}
\frac{d}{dt} \Phi(y',t) = \mu(\Phi(y',t)), & \forall t\neq 0,\\
\Phi(y',0) = (y',0). & 
\end{cases}
\end{aligned}
\end{equation}
We want to use $\Phi(y',t)$ to introduce a new system of coordinates, that is, for all $x \in \RR^n$, we want to show that there is a unique $(y',t)\in\RR^n$ such that $x = \Phi(y',t)$. For this purpose, let $G:\RR^n\times\RR^n\rightarrow\RR^n$ be defined by
$$
G(x,y',t) = x - \Phi(y',t),\quad \forall x \in \RR^n,\ \forall (y',t)\in\RR^n.
$$
From our assumption that $\nu_n(O)\neq 0$, we also have that $\mu_n(O)\neq 0$. We see that
$$
|J_{(y',t)} \Phi(O,O)| \neq 0,
$$
and so, the Implicit Function Theorem gives that there are neighborhoods of $O$, which we denote by $W$ and $V$, such that for all $x \in W$, there is a unique $(y',t) \in V$ such that 
\begin{equation}
\label{eq:1to1_correspondence}
x = \Phi(y',t).
\end{equation} By choosing $r$ small enough, we can make the `oscillation' of the vector field $\mu$ small enough, so that the proof of the Implicit Function Theorem implies that $B_{2r}(O) \subset W$. Because the vector field $\mu$ is constant outside the ball $B_{2r}(O)$ by identity \eqref{eq:New_drift}, we can extended the 1-1 correspondence \eqref{eq:1to1_correspondence}, between points $x\in W$ and $(y',t)\in V$, to all points $x\in\RR^n$ and $(y',t)\in\RR^n$. To see this, let
$$
S=\partial W \cup (\{x_n=0\}\backslash W).
$$ 
For each $x \in W^c$, let $x' \in S$ be the closest point to $x$, with respect to the Euclidean distance, with the property that $x-x'$ and $\nu(O)$ are linearly dependent. Because $x'\in S$, there is $(y',t')\in V$ such that $x'=\Phi(y',t')$. Now let $t$ be such that
$$
(t-t') \nu(O) = x-x',
$$
then $x=\Phi(y',t)$. Therefore, the function $\Phi:\RR^n\rightarrow\RR^n$ is a bijection.

Because $\mu\in \cC^{\infty}_b(\RR^n;\RR^n)$, \cite[Corollary V.4.1]{Hartman} shows that $\Phi$ belongs to $\cC^{\infty}_b(\RR^n;\RR^n)$. We notice that 
$$
|J\Phi(y',0)| \neq 0, \quad\forall y'\in\RR^{n-1}.
$$
Because $\mu$ is a constant vector field outside the ball $B_{2r}(O)$, we can choose $r$ small enough so that
\begin{equation}
\label{eq:Non_zero_Jacobian}
|J\Phi(y',t)| \neq 0, \quad\forall (y',t)\in\RR^n,
\end{equation}
and we can also arrange so that condition \eqref{eq:Diffeomorphism_nondegeneracy} and assumption \eqref{item:Auxiliary_function_invertibility_smoothness} of Lemma \ref{lem:Change_of_coordinates} hold.

From \eqref{eq:Non_zero_Jacobian}, the Inverse Function Theorem now shows that $\Phi^{-1}$ is a $C^1$ function on $\RR^n$ with bounded derivatives of first order. From the identity,
$$
J\Phi^{-1} = (J\Phi \circ \Phi^{-1})^{-1},
$$
we see that, the fact that the function $\Phi$ is smooth and the inverse function $\Phi^{-1}$ is $C^1$, and both have bounded derivatives, implies that $\Phi^{-1}$ is a $C^2$ function with bounded derivatives up to order two. Inductively, we obtain that $\Phi^{-1}$ is a $C^{\infty}$ function with bounded derivatives of any order.

Moreover, for any function $u\in \cC^1(\RR^n)$, applying the change of variable $v=\Phi^*u$ and $x=\Phi(y)$, we have that
\begin{align*}
v_{y_n}(y) &= \sum_{i=1}^n \frac{\partial u}{\partial x_i}(\Phi(y)) \frac{\partial \Phi^i}{\partial y_n}(y)\\
&= \mu(x) \dotprod \nabla u(x)\quad\hbox{(by \eqref{eq:ODE}, and the fact that $x=\Phi(y)$).}
\end{align*}
Therefore, identity \eqref{eq:Derivatives_change_of_coordinates} is satisfied \emph{locally}, for all $x\in B_{r}(O)$, because $\mu(x)=\nu(x)$ on $B_r(O)$ by construction.

The preceding argument proves
\begin{lem}[Construction of a diffeomorphism with suitable local properties]
\label{lem:Construction_diffeomorphism_local}
Let $x_0 \in \RR^n$ be such that $\nu(x_0) \neq 0$, and let $\nu\in \cC^{\infty}(\RR^n;\RR^n)$. Then there is diffeomorphism, $\Phi:\RR^n\rightarrow\RR^n$, which satisfies assumptions \eqref{item:Diffeomorphism_invertibility_smoothness} and \eqref{item:Auxiliary_function_invertibility_smoothness} of Lemma \ref{lem:Change_of_coordinates}, and verifies condition \eqref{eq:Diffeomorphism_nondegeneracy} of Lemma \ref{lem:Parametrix_b}, and there is a positive constant, $r$, such that identity \eqref{eq:Derivatives_change_of_coordinates} of Lemma \ref{lem:Change_of_coordinates_cut_symbol} holds, for all $x \in B_r(x_0)$.
\end{lem}

\subsubsection{Localization}
\label{subsubsec:Localization}
Because we are interested in the local regularity of $u$, we only need to study $\chi u$, where $\chi:\RR^n\rightarrow [0,1]$ is a smooth function with compact support.

\begin{rmk}[Comparison between Lemma \ref{lem:Regularity_solutions_general_case} and Theorem \ref{thm:Local_regularity}]
The difference between Lemma \ref{lem:Regularity_solutions_general_case} and Theorem \ref{thm:Local_regularity} is that Lemma  \ref{lem:Regularity_solutions_general_case} assumes the existence of a diffeomorphism, $\Phi$, satisfying suitable properties, while in Theorem \ref{thm:Local_regularity} this assumption is replaced by the condition that $\nu(x_0)\neq 0$. The result that allows us to weaken the hypotheses of Lemma \ref{lem:Regularity_solutions_general_case}, to only assume that $\nu(x_0)\neq 0$ in the statement of Theorem \ref{thm:Local_regularity}, is Lemma \ref{lem:Construction_diffeomorphism_local}.
\end{rmk}

\begin{proof}[Proof of Theorem \ref{thm:Local_regularity}]
For clarity, we divide the proof into three steps. The first step is an application of Lemma \ref{lem:Construction_diffeomorphism_local}, which allows us to locally change the coordinates so that we can construct a two-sided parametrix for the pseudo-differential operator in the new system of coordinates. In the second step, we collect various intermediate results which are ingredients in the iteration procedure employed to obtain the local regularity of solutions. The iteration procedure is described in the third step.

\setcounter{step}{0}
\begin{step}[Change of coordinates]
\label{step:Change_of_coordinates}
Without loss of generality, we may assume that $r=1$, that is $\psi_r=\psi$,
where we recall that the cut-off function $\psi_r$ is defined in
\eqref{eq:Psi_r} and $\psi$ is defined by \eqref{eq:Psi}. Therefore, we have
that $\psi Au \in H^l(\RR^n)$. Because we assume that $\nu(x_0)\neq 0$, we may
choose a system of coordinates such that $x_0=O$ and $\nu(O)\cdot e_n=\nu_n(O)\neq 0$. By Lemma \ref{lem:Construction_diffeomorphism_local} there is a diffeomorphism, $\Phi:\RR^n\rightarrow\RR^n$, which satisfies assumptions \eqref{item:Diffeomorphism_invertibility_smoothness} and \eqref{item:Auxiliary_function_invertibility_smoothness} of Lemma \ref{lem:Change_of_coordinates}, and verifies condition \eqref{eq:Diffeomorphism_nondegeneracy} of Lemma \ref{lem:Parametrix_b}. Moreover, there is a positive constant, $\bar r$, such that identity \eqref{eq:Derivatives_change_of_coordinates} of Lemma \ref{lem:Change_of_coordinates_cut_symbol} holds, for all $x \in B_{\bar r}(x_0)$. Let $r_0:=(1\wedge \bar r)/4$, and let the vector field $\mu$ be defined by \eqref{eq:New_drift}, where the constant $r$ is replaced by $\bar r$. We consider the new symbol, $\alpha(x,\xi)$, obtained from $\tilde a$ by replacing the drift coefficient $\nu$ by $\mu$, that is
$$
\alpha(x,\xi)=|\xi|^{2s} \varphi(2 \xi) + i \mu(x)\dotprod\xi,\quad\forall x,\xi\in\RR^n.
$$
Then the diffeomorphism $\Phi$ satisfies the hypotheses of Lemma \ref{lem:Change_of_coordinates_cut_symbol}, with the vector field $\nu$ replaced by $\mu$. We obtain that there is a symbol $\beta \in S^{1}_{1,0}(\RR^n\times\RR^n)$ such that
\begin{equation}
\label{eq:Change_of_coordinates_alpha}
\Phi^* T_{\alpha} (\Phi^{-1})^* = T_{\beta}.
\end{equation}
\end{step}

\begin{step}[Intermediate results]
\label{step:Intermediary_results} 
Since $\psi\in\cC^{\infty}_c(\RR^n),$ there is a real $k$ so that $\psi u\in H^k(\RR^n).$
Let $J\geq 1$ be the smallest integer such that $k+2sJ \geq l$. We choose a family of smooth cut-off functions, $\{\chi_j:j=1,\ldots, J+1\}$, with values in $[0,1]$ such that 
\begin{align}
\label{eq:chi_j_equal_1}
&\chi_j = 1\hbox{ on } B_{r_0}(O),\quad\forall j=1,\ldots, J+1, \\
\label{eq:chi_j_equal_0}
&\chi_j = 0\hbox{ on } B^c_{2r_0}(O),\quad\forall j=1,\ldots, J+1,\\
\label{eq:_support_chi_j}
&\hbox{supp } \chi_{j+1} \subseteq \{ \chi_j=1\},\quad\forall j=1,\ldots,J.
\end{align}
From \cite[Property (4.2.19)]{Taylor_vol1} and using the fact that $\hbox{supp } \chi_j \subset \{\psi=1\}$, it follows that 
\begin{equation}
\label{eq:Cutoff_u_Sobolev}
\chi_j u\in H^k(\RR^n),\quad \forall j=1,\ldots, J+1.
\end{equation}
Our goal is to show that
\begin{equation}
\label{eq:cutoff_T_alpha_cutoff}
\chi_{j+1}T_{\alpha}\chi_j u \in H^l(\RR^n),\quad \forall j=1,\ldots,J.
\end{equation}
Because $\nu=\mu$ on $B_{2r_0}(O)$, and the support of each $\chi_j$ is contained in $B_{2r_0}(O)$, we have
$$
\chi_j T_{\alpha} u = \chi_j T_{\tilde a} u,\quad\forall j=1,\ldots, J+1.
$$
We recall that $T_{\tilde a} u = Au + Eu$, where $Eu$ is defined as in
\eqref{eq:Definition_E}. From our assumption that $\psi Au \in H^l(\RR^n)$, we
have that $\chi_jAu$ is in $H^l(\RR^n)$, by \eqref{eq:_support_chi_j},
definition \eqref{eq:Psi} of $\psi$, and the choice of the constant $r_0$. Because we assume
that $u\in \cS'_{(0,s)}(\RR^n)$, we know that $Eu \in \cC^{\infty}_t(\RR^n),$ from
\eqref{eq:Definition_E_with_domain_distrib}. Therefore, we can conclude that
\begin{equation}
\label{eq:cutoff_T_alpha}
\chi_{j} T_{\alpha} u \in H^l(\RR^n),\quad \forall j=1,\ldots, J+1.
\end{equation}
We can write
$$
\chi_{j+1} T_{\alpha} u = \chi_{j+1} T_{\alpha} \chi_j u  + \chi_{j+1} T_{\alpha} (1-\chi_j) u, \quad \forall j=1,\ldots, J.
$$
The preceding identity together with \eqref{eq:cutoff_T_alpha} shows that to obtain \eqref{eq:cutoff_T_alpha_cutoff}, it is enough to establish 
\begin{equation}
\label{eq:T_alpha_cutoff}
\chi_{j+1}T_{\alpha} (1-\chi_j) u \in H^l(\RR^n),\quad \forall j=1,\ldots, J+1.
\end{equation}
We can write $T_{\alpha}=T_{\alpha^1}+T_{\alpha^2}$ as the sum of two pseudo-differential operators with symbols
$$
\alpha^1(\xi)=|\xi|^{2s}\varphi(2\xi)\quad\hbox{and}\quad \alpha^2(x,\xi)=i\mu(x)\dotprod\xi,\quad\forall x,\xi\in \RR^n.
$$
We see that
$$
T_{\alpha^2} v(x) = \mu(x)\dotprod\nabla v(x),\quad\forall v \in \cS'(\RR^n).
$$
From \eqref{eq:_support_chi_j}, we obtain that $\chi_{j+1} T_{\alpha^2}
(1-\chi_j) u =0$ on $\RR^n$. We also see that $T_{\alpha^1}$ is a classical
pseudo-differential operator, and hence pseudolocal, so \cite[Proposition
7.4.1]{Taylor_vol2} implies that the function $\chi_{j+1} T_{\alpha^1}
(1-\chi_j) u$ is smooth with compact support, hence contained in
$H^l(\RR^n)$. Because both \eqref{eq:cutoff_T_alpha} and
\eqref{eq:T_alpha_cutoff} hold, we obtain that property
\eqref{eq:cutoff_T_alpha_cutoff} holds.
\end{step}

\begin{step}[Iteration procedure]
\label{step:Iteration_procedure}
We now employ an iteration procedure to prove the local regularity of
solutions. We denote $w_j=\chi_j u$, where the cut-off functions,
$\{\chi_j:j=1,\ldots, J+1\}$, are chosen in Step
\ref{step:Intermediary_results}. We consider the change of coordinates
$w_j(x)=v_j(y)$ and $x=\Phi(y)$. From identity
\eqref{eq:Change_of_coordinates_alpha}, we have that
$$
T_{\beta} v_{j+1}(y)=  \Phi^* T_{\alpha} w_{j+1}(y),\quad\forall j=1,\ldots, J.
$$
Any smooth function with compact support, $\chi:\RR^n\rightarrow [0,1]$, can be
viewed as a symbol in the class $S^0_{1,0}(\RR^n\times\RR^n)$. Because $\alpha$
is a symbol in $S^1_{1,0}(\RR^n\times\RR^n)$, the commutator rule \cite[Chapter
7, Identity (3.24)]{Taylor_vol2} gives that there is symbol, $e \in
S^0_{1,0}(\RR^n\times\RR^n)$, such that
\begin{equation}
\label{eq:RHS_T_alpha}
T_{\alpha} (\chi u) = \chi T_{\alpha} u + T_e u,\quad\forall u \in \cS'(\RR^n).
\end{equation}
Using the fact that
\begin{align*}
T_{\alpha} w_{j+1}(y) &= T_{\alpha} \chi_{j+1}\chi_j u(x)\\
&=\chi_{j+1} T_{\alpha}\chi_j u(x) + T_e \chi_j u(x) \quad\hbox{(by identity \eqref{eq:RHS_T_alpha})} \\
&=\chi_{j+1} T_{\alpha}w_j(x)  + T_e w_j(x),
\end{align*}
it follows that
\begin{equation}
\label{eq:Induction_relation}
T_{\beta} v_{j+1}(y)= \Phi^* (\chi_{j+1} T_{\alpha} w_{j} + T_e w_j)(y),\quad\forall j=1,\ldots, J.
\end{equation}

We now prove inductively that
\begin{equation}
\label{eq:Induction_conclusion}
v_{j+1}  \in H^{l\wedge (k + (j-1)2s)+2s}(\RR^n).
\end{equation}
If $j=1$, we have shown in \eqref{eq:cutoff_T_alpha_cutoff} that the function
$\chi_{j+1} T_{\alpha} w_j$ belongs to $H^l(\RR^n)$, and because the symbol $e
\in S^0_{1,0}(\RR^n\times\RR^n)$ and the function $w_j \in H^k(\RR^n)$ by
\eqref{eq:Cutoff_u_Sobolev}, we have that $ T_e w_j$ belongs to
$H^k(\RR^n)$. Therefore, the function $\chi_{j+1} T_{\alpha} w_{j}(y) + T_e
w_j$ belongs to $H^{l\wedge k}(\RR^n)$. Applying Proposition
\ref{prop:Sobolev_spaces_diffeomorphism}, we obtain that the right-hand side in
identity \eqref{eq:Induction_relation} also belongs to $H^{l\wedge k}(\RR^n)$. By Lemma \ref{lem:Parametrix_b}, the operator $T_{\beta}$ admits
a left-parametrix, $T_p$, with symbol $p \in
S^{-2s}_{2s,0}(\RR^n\times\RR^n)$. Thus, it follows that property
\eqref{eq:Induction_conclusion} holds when $j=1$.

Assume now that property \eqref{eq:Induction_conclusion} holds for $j=j_0 \geq
1$. We want to prove that \eqref{eq:Induction_conclusion} holds for
$j=j_0+1$. Using the fact that
$$
v_{j_0+1}  \in H^{l\wedge (k + (j_0-1)2s)+2s}(\RR^n),
$$
and that $w_{j_0+1}=(\Phi^{-1})^* v_{j_0+1}$, we obtain from Proposition
\ref{prop:Sobolev_spaces_diffeomorphism} that
$$
w_{j_0+1}  \in H^{l\wedge (k + (j_0-1)2s)+2s}(\RR^n).
$$
To prove that property \eqref{eq:Induction_conclusion} holds in the case
$j=j_0+1$, we can apply the same argument that we employed to prove
\eqref{eq:Induction_conclusion} in the case when $j=1$, with the observation
that we replace the role of $w_1$ with that of $w_{j_0+1}$, and the role of $k$
is replaced with that of $l\wedge (k + (j_0-1)2s)+2s$.

When $j=J+1$, we see that $k + 2s J \geq l$, and so $l\wedge (k + (j-1)2s) = l$. From \eqref{eq:Induction_conclusion}, it follows that
$$
v_{J+1}  \in H^{l+2s}(\RR^n).
$$
As $T_{\beta}v_{J+1}\in H^{l}(\RR^n),$ we easily conclude, as before, that 
\begin{equation}
  \partial_{y_n}v_{J+1}\in H^{l}(\RR^n).
\end{equation}
Recall that 
$$
w_{J+1}=(\Phi^{-1})^*v_{J+1}\quad\hbox{and}\quad \nu\dotprod\nabla w_{J+1} =(\Phi^{-1})^*\partial_{y_n} v_{J+1}.
$$
Proposition \ref{prop:Sobolev_spaces_diffeomorphism} implies that $w_{J+1}$ is
in $H^{l+2s}(\RR^n)$ and $\nu\dotprod\nabla w_{J+1}$ is in $H^{l}(\RR^n)$, and
since $\chi_{J+1}$ can be any smooth function with values in $[0,1]$ and
compact support in $B_{r_0}(O)$, the conclusion follows immediately.
\end{step}
This concludes the proof.
\end{proof}

\section{The Green's Kernel}
In~\cite{Bogdan_Jakubowski_2012} and \cite{Chen_Kim_Song_2012} estimates are
given for the Green's kernel on a bounded domain, assuming that $\frac 12<s<1,$
but allowing a somewhat singular coefficient for the vector field. In our
setting for this range of $s,$ with a smooth coefficient for the vector field,
the operator $A$ is a classical, elliptic pseudodifferential operator of order
$2s,$ so the calculation of the leading singularity of the Green's function is
trivial.  In the previous sections we employed a change of coordinates
$y=\Phi(x),$ to obtain a parametrix for the fundamental solution for the
operators $Au(x)=(-\Delta)^su(x)+\nu(x)\cdot\nabla u(x),$ for $0<s<\frac 12,$
under the assumption that $\nu(x)$ is non-vanishing. In this section we examine
the leading term in the asymptotic expansion, in the ``$y$''-coordinates, of
this operator's kernel along the diagonal.

The kernel of the leading term, after
applying the change of variables, $y=\Phi(x),$ is given by
\begin{equation}
  K(y,z)=\frac{1}{(2\pi)^n}\int\limits_{\RR^n}\frac{\psi(\eta)e^{i\eta\cdot
      z}d\eta}
{|(J\Phi(y))^{-t}\eta|^{2s}+i\eta_n},
\end{equation}
where $z=y-\ty$ as usual.
After changing variables with $(J\Phi(y))^{t}\chi= \eta,$ this becomes
\begin{equation}
  K(y,z)=\frac{1}{(2\pi)^n}\int\limits_{\RR^n}\frac{\psi((J\Phi(y))^{t}\chi)e^{i\chi\cdot
      J\Phi(y) z}|J\Phi(y)|d\chi}
{|\chi|^{2s}+i((J\Phi(y))^{t}\chi)_n}.
\end{equation}
We finally choose an orthogonal transformation $U(y)$ so that
$b(y)(U(y)\chi)_n=((J\Phi(y))^{t}\chi)_n,$ where we normalize so that $b(y)>0.$
Setting $\xi= U(y)\chi,$ the integral becomes
\begin{equation}\label{eqn4.3.001}
  K(y,z)=\frac{1}{(2\pi)^n}\int\limits_{\RR^n}\frac{\psi((J\Phi(y))^{t}U(y)^t\xi)e^{i\xi\cdot
      U(y)J\Phi(y) z}|J\Phi(y)|d\xi}
{|\xi|^{2s}+ib(y)\xi_n}.
\end{equation}
Up to a linear transformation, the leading order singularities in $K(y,z)$
as $z\to 0,$  are determined by asymptotic  evaluations  of the Fourier
transforms
\begin{equation}\label{eqn.4.001}
  E_{s,b}(x)=\int\limits_{\RR^n}\frac{e^{ix\cdot\xi}d\xi}{|\xi|^{2s}+ib\xi_n},
\end{equation}
as $x\to 0.$ Here $b$ is a positive constant. In fact,
equation~\eqref{eqn4.3.001} shows that, up to a smooth error:
\begin{equation}
  K(y,z)=E_{s,b(y)}(U(y)J\Phi(y) z)\cdot |J\Phi(y)|.
\end{equation}
 $E_{s,b}(x)$ is the kernel for
Green's function of 
the constant coefficient operator $(-\Delta_{\RR^n})^{s}+b\partial_{x_n}.$ 
The ``Green's function'' for the operator
$(-\Delta_{\RR^{n-1}})^{s}+b\partial_{x_n}$ is the kernel, $p^0_s(x';x_n),$ of
the heat operator $e^{-\frac{x_n}{b}(-\Delta_{\RR^{n-1}})^{s}},$ which is supported in
the half space $\{(x',x_n):\:0<x_n\}.$ As we shall see, when $0<s<\frac 12,$
there is an echo of this in the kernel for $E_{s,b}(x);$ its leading term is
\begin{equation}
  O\left(\frac{1}{[x_n^{\frac 1s}+|x'|^2]^{\frac{n-1}{2}}}\right)\text{ where }x_n>0,
\end{equation}
 and
 \begin{equation}
   O\left(\frac{1}{[x_n^2+|x'|^2]^{\frac{n-2+2s}{2}}}\right)\text{ where }x_n<0.
 \end{equation}
Note that when $2s<1,$ then $n-2+2s<n-1.$ Notice also that the homogeneities of
these terms are different.

\subsection{Asymptotics of $E_{s,b}$}
To evaluate the asymptotic behavior of the integral in~\eqref{eqn.4.001}, as
$|x|\to 0,$  we split it into an integral over the first $n-1$
variables and an integral over $\xi_n:$
\begin{equation}\label{eqn4.5.0002}
\begin{split}
  E_{s,b}(x)&=\int\limits_{\RR^{n-1}}\int\limits_{-\infty}^{\infty}
\frac{e^{ix'\cdot\xi'}e^{ix_n\xi_n}d\xi_n d\xi'}{(|\xi'|^2+\xi_n^2)^{s}+ib\xi_n}\\
  &=\omega_{n-2}\int\limits_{0}^{\infty}
\int\limits_{0}^{\pi}\int\limits_{-\infty}^{\infty}\frac{
  e^{ix_n\xi_n} d\xi_ne^{ir|x'|\cos\theta}\sin^{n-3}\theta
d\theta r^{n-2}dr}{(r^2+\xi_n^2)^{s}+ib\xi_n}\\
&=c_n\omega_{n-2}
  \int\limits_{0}^{\infty}\int\limits_{-\infty}^{\infty}\frac{e^{ix_n\xi_n}d\xi_n}
{(r^2+\xi_n^2)^{s}+ib\xi_n}
\frac{J_{\frac{n-3}{2}}(r|x'|) r^{n-2}dr }{(r|x'|)^{\frac{n-3}{2}}},
\end{split}
\end{equation}
where $x=(x',x_n),\,\xi=(\xi',\xi_n),$ and 
\begin{equation}
  c_n=\sqrt{\pi}2^{\frac{n-3}{2}}\Gamma\left(\frac{n-2}{2}\right).
\end{equation}
To simplify the computation we let
\begin{equation}
  \xi_n=r\tau,\text{ and }\beta=br^{1-2s},
\end{equation}
obtaining
\begin{equation}\label{eqn4.5}
  E_{s,b}(x)=
  \frac{c_n\omega_{n-2}}{|x'|^{\frac{n-3}{2}}}
  \int\limits_{0}^{\infty}\int\limits_{-\infty}^{\infty}
  \frac{e^{irx_n\tau}d\tau J_{\frac{n-3}{2}}(r|x'|) r^{\frac{n-1}{2}}r^{1-2s}dr }{(1+\tau^2)^{s}+i\beta\tau}.
\end{equation}
To evaluate the $\tau$-integral, we consider the function
$\frac{1}{(1+\tau^2)^{s}+i\beta\tau}$ to be a single valued analytic function in the slit
region
\begin{equation}
  D=\CC\setminus (-i\infty,-i]\cup [i,i\infty).
\end{equation}
A careful examination of the denominator on the boundary of $D$ reveals that it
has a single pole. This pole is of the form $\tau =iy(\beta),$ where
$y(\beta)\in (0,1)$ satisfies the equation:
\begin{equation}
  (1-y^2(\beta))^s=\beta y(\beta).
\end{equation}
As $\beta\to 0,$ we see that
\begin{equation}
  y(\beta)=1-\frac{\beta^{\frac{1}{s}}}{2}\left(1+\sum_{j=1}^{\infty}a_j\beta^{\frac{j}{s}}\right),
\end{equation}
where the sum is convergent in some neighborhood of $\beta=0.$ As
$\beta\to\infty,$ we can show that
\begin{equation}\label{eqn4.15.004}
  y(\beta)=\frac{1}{\beta}\left(1+\sum_{j=1}^{\infty}\frac{b_j}{\beta^{2j}}\right),
\end{equation}
where again the series in convergent for $1/\beta$ in a neighborhood of
zero. 

To compute the $\tau$-integral for $x_n>0$ we use the contour $\Gamma_R^+$
shown in Figure~\ref{fig1};  letting $R$ go to infinity, we easily show that
\begin{equation}\label{eqn4.11}
  \lim_{R\to\infty}\int\limits_{-R}^{R}\frac{r^{1-2s}\ e^{irx_n\tau}
    d\tau}{(1+\tau^2)^{s}+i\beta\tau}=
\frac{2\pi e^{-rx_n y(\beta)}(1-y^2(\beta))}{b(1-(1-2s)y^2(\beta))}+
4\pi r^{1-2s}\int\limits_{1}^{\infty}\frac{e^{-rx_n\tau}\sin(\pi s)(\tau^2-1)^s
  d\tau}{|(\tau^2-1)^se^{\pi i s}-\beta\tau|^2}.
\end{equation}
\begin{figure}[H]
\centering
{\epsfig{file=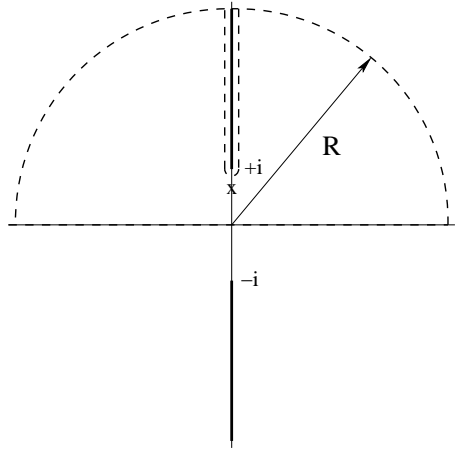,width=6cm}}
\caption{The cut-plane where the $\tau$-integrand is single valued showing the
  contour $\Gamma^+_R$ and the location of the pole.}
\label{fig1}
\end{figure}
A similar calculation using the reflection of $\Gamma^+_{R}$ across the real
axis gives that for $x_n<0,$
\begin{equation}\label{eqn4.12}
  \lim_{R\to\infty}\int\limits_{-R}^{R}\frac{r^{1-2s} e^{irx_n\tau}
    d\tau}{(1+\tau^2)^{s}+i\beta\tau}=
4\pi r^{1-2s}\int\limits_{1}^{\infty}\frac{e^{-r|x_n|\tau}\sin(\pi s)(\tau^2-1)^s
  d\tau}{|(\tau^2-1)^se^{\pi i s}+
\beta\tau|^2}.
\end{equation}
Recall that $\beta=br^{1-2s},$ and $y(\beta)$ has the expansion
in~\eqref{eqn4.15.004}, so, as $r\to\infty,$ the residue term
in~\eqref{eqn4.11} behaves like
\begin{equation}\label{eqn4.18.004}
  \frac{2\pi e^{-r^{2s}x_n/b}}{b},
\end{equation}
which decays at a slower exponential rate than the contribution from the
integral. To simplify the formulae that follow we let
$e^{-r|x_n|}I^{\pm}_s(r,x_n)$ denote the integral terms on the right hand sides
of~\eqref{eqn4.11}
and~\eqref{eqn4.12}, and  $H(r,x_n)$ denote the residue term
in~\eqref{eqn4.11} where $\beta=br^{1-2s}.$ With this notation we see that
\begin{equation}
  E_{s,b}(x',x_n)=\begin{cases}
&\frac{c_n\omega_{n-2}}{|x'|^{\frac{n-3}{2}}}
  \int\limits_{0}^{\infty}
J_{\frac{n-3}{2}}(r|x'|) \left[ e^{-r|x_n|} I^{+}_s(r,x_n)+H(r,x_n)\right]
 r^{\frac{n-1}{2}} dr\, \text{ if }x_n>0,\\
&\frac{c_n\omega_{n-2}}{|x'|^{\frac{n-3}{2}}}
  \int\limits_{0}^{\infty}
J_{\frac{n-3}{2}}(r|x'|) e^{-r|x_n|} I^{-}_s(r,x_n)
 r^{\frac{n-1}{2}} dr\,\text{ if }x_n<0.
\end{cases}
\end{equation}
This explains why, if $2s<1,$ $E_{s,b}(x',x_n)$ has a stronger singularity
approaching $(0,0)$ from $x_n>0$ than from $x_n<0.$ Where $x_n>0$ the principal
term in $E_{s,b}(x',x_n)$ is essentially a multiple of $p^0_s(x';x_n).$

\subsection{Asymptotics of the Integral Terms}

To understand the behavior of the integral terms when $r$ is small, we let
$\sigma=r(\tau-1),$ to obtain that
\begin{equation}
  I^{\pm}_s(r,x_n)=4\pi \int\limits_{0}^{\infty}\frac{\sin(\pi
    s)e^{-|x_n|\sigma}\sigma^s(\sigma+2r)^s d\sigma}
{|\sigma^s(\sigma+2r)^se^{\pi i s}
\mp b(\sigma+r)|^2}.
\end{equation}
From these expressions and the fact that $2s<1,$ it is immediate that the limits
\begin{equation}\label{4.21.004}
 C_s^{\pm}= \lim_{(r,x_n)\to (0^+,0^{\pm})}I^{\pm}_s(r,x_n)
\end{equation}
are finite. Hence from~\eqref{eqn4.18.004}, \eqref{4.21.004}, and the
asymptotic behavior of Bessel functions near zero  we see that, for any finite $r_0,$ the integrals
\begin{equation}
  \frac{c_n\omega_{n-2}}{|x'|^{\frac{n-3}{2}}}
  \int\limits_{0}^{r_0}
J_{\frac{n-3}{2}}(r|x'|) r^{\frac{n-1}{2}} e^{-r|x_n|} I^{\pm}_s(r,x_n)dr,
\quad \frac{c_n\omega_{n-2}}{|x'|^{\frac{n-3}{2}}}
  \int\limits_{0}^{r_0}
J_{\frac{n-3}{2}}(r|x'|) r^{\frac{n-1}{2}} H(r,x_n)dr  
\end{equation}
remain bounded as $(|x'|,x_n)$ tends to zero. 

In determining the asymptotics of $E_{s,b}(x',x_n)$ we are therefore free to
restrict attention in~\eqref{eqn4.5} to $0<<r_0<r.$ In this regime we 
rewrite $I^{\pm}_s$ as
\begin{equation}
  I^{\pm}_s(r,x_n)= 4\pi \int\limits_{r}^{\infty}\frac{\sin(\pi
    s)e^{-|x_n|(\tau-r)}(\tau^2-r^2)^s d\tau}
{|(\tau^2-r^2)^se^{\pi i s}\mp b\tau|^2}.
\end{equation}
Expanding these integrands in powers of
\begin{equation}
  \frac{(\tau^2-r^2)^s}{b\tau},
\end{equation}
gives the convergent expansions:
\begin{equation}
  \frac{(\tau^2-r^2)^s }
{|(\tau^2-r^2)^se^{\pi i s}\mp b\tau|^2}=
\frac{(\tau^2-r^2)^s }
{(b\tau)^2}\left[1+\sum_{j=1}^{\infty}a_j^{\pm}\left( \frac{(\tau^2-r^2)^s}{b\tau}\right)^j\right],
\end{equation}
valid for sufficiently large $\tau.$
The functions $I_s^{\pm}(r,x_n),$ then have expansions of the form
\begin{equation}
  I_s^{\pm}(r,x_n)\sim \sum_{j=0}^{\infty}a_j^{\pm}\frac{f_j(r|x_n|)}{r^{(1-2s)(j+1)}},
\end{equation}
where
\begin{equation}
  f_j(y)=\int\limits_{1}^{\infty}e^{-y(\sigma-1)}\left[\frac{(\sigma^2-1)^s}{\sigma}\right]^{j+1}\frac{d\sigma}{\sigma}.
\end{equation}

Using a standard remainder estimate for a geometric series, it follows easily
that for given $0<b$ and $0<s<\frac 12$ there is an $r_0$ so that, for each $N$ there is a $C_N,$ for which
\begin{equation}\label{eqn4.24.002}
  \left|I_s^{\pm}(r,x_n)- \sum_{j=0}^{N}a_j^{\pm}\frac{f_j(r|x_n|)}{r^{(1-2s)(j+1)}}\right|
\leq\frac{C_N}{r^{(1-2s)(N+2)}},\text{ if }r_0\leq r.
\end{equation}
The functions $\{f_j\}$ are non-negative, and $f_j(0)$ is  positive. As
$y\to\infty,$ $f_j(y)$ behaves like a classical symbol.
\begin{lem}\label{fjsymbest}
  For each $j\in\NN_0$ the function $f_j(y)$ is in $\cC^{\infty}((0,\infty))$ and for
  each $k\in \NN_0$ there is a constant $C_{j,k}$ so that, 
  \begin{equation}
    |\partial_y^kf_j(y)|\leq\frac{C_{k,j}}{y^{k+1+s(j+1)}}\text{ for }y\in [1,\infty).
  \end{equation}
\end{lem}
\begin{proof}
Letting $\tau=\sigma-1,$ we can rewrite $f_j$ as
  \begin{equation}\label{eqn4.30.004}
    f_j(y)=\int\limits_{0}^{\infty}e^{-y\tau}\left[\frac{\tau^s(\tau+2)^s}{\tau+1}\right]^{j+1}
\frac{d\tau}{\tau+1}.
  \end{equation}
We can differentiate under the integral sign to obtain
\begin{equation}
 \partial_y^k f_j(y)=(-1)^k\int\limits_{0}^{\infty}e^{-y\tau}\tau^k
\left[\frac{\tau^s(\tau+2)^s}{\tau+1}\right]^{j+1}
\frac{d\tau}{\tau+1}.
\end{equation}
If we let $y\tau=w,$ then we see that
\begin{equation}
 \partial_y^k f_j(y)=\frac{(-1)^k}{y^{k+1+s(j+1)}}\int\limits_{0}^{\infty}e^{-w}w^{k+s(j+1)}
\left[\frac{(2+w/y)^s}{1+w/y}\right]^{j+1}
\frac{dw}{1+w/y}.
\end{equation}
The Lebesgue dominated convergence theorem easily implies that
\begin{equation}
  \lim_{y\to\infty} y^{k+1+s(j+1)}\partial_y^k f_j(y)=(-1)^k2^{s{j+1}}
\int\limits_{0}^{\infty}e^{-w}w^{k+s(j+1)}dw<\infty,
\end{equation}
from which the lemma follows easily.  
\end{proof}

We also have the following result describing the behavior of $f_j(y)$ as
$y\to 0^+.$
\begin{lem}\label{lem4.1.002}
  For $j\geq 0,$ if $(j+1)s\notin\NN,$ then there are functions
  $f_{j0}(y), f_{j1}(y)\in\cC^{\infty}[0,\infty)$ so that
  \begin{equation}\label{eqn4.30.004}
f_j(y)=f_{j1}(y)+f_{j0}(y)y^{(1-2s)(j+1)}.
\end{equation}
\end{lem}

\begin{rmk} Throughout the remainder of this section we assume that
  $s\notin\QQ.$ This is not an essential restriction, but allows us to avoid
  considering many special cases.
\end{rmk}

\begin{proof}[Proof of Lemma \ref{lem4.1.002}] 
  We use formula~\eqref{eqn4.30.004} for $f_j(y).$ It is clear that the
  function is smooth for $y\in (0,\infty)$ and that the possible singularities
  as $y\to 0^+$ only result from the large $\tau$ behavior of the
  integrand. Thus it suffices to show that
\begin{equation}
    f^+_j(y)=\int\limits_{10}^{\infty}e^{-y\tau}\left[\frac{\tau^s(\tau+2)^s}{\tau+1}\right]^{j+1}\frac{d\tau}{\tau+1}
\end{equation}
has the behavior given in~\eqref{eqn4.30.004}. To that end we rewrite this
integral in the form
\begin{equation}
\begin{split}
  f^+_j(y)&=\int\limits_{10}^{\infty}e^{-y\tau}\tau^{(2s-1)(j+1)}\left[\frac{(1+2/\tau)^s}{1+1/\tau}\right]^{j+1}\frac{d\tau}{\tau(1+1/\tau)}\\
&=\sum_{k=1}^{\infty}a_k\int\limits_{10}^{\infty}e^{-y\tau}\tau^{(2s-1)(j+1)}\tau^{-k}d\tau.
\end{split}
\end{equation}
To prove the lemma it therefore suffices to examine functions of the form
\begin{equation}
  \int\limits_{10}^{\infty}e^{-y\tau}\tau^{-\alpha-l}d\tau,
\end{equation}
where $0<\alpha<1,$ and $l\in\NN.$ Repeatedly integrating by parts we can show
that there is a constant $c_{l,\alpha}$ so that
\begin{equation}
  \int\limits_{10}^{\infty}e^{-y\tau}\tau^{-\alpha-l}d\tau=c_{l,\alpha}y^{l+\alpha-1}+g_{l,\alpha}(y),
\end{equation}
where $g_{l,\alpha}\in\cC^{\infty}([0,\infty)).$ Using this relation
repeatedly, with $l+\alpha=(1-2s)(j+1)+k,$ and noting that $k\geq 1,$ we easily obtain
the assertion of the lemma.
\end{proof}

We now derive the leading order asymptotics of $E_{s,b}(x)$ of $x$ as $x\to 0.$
We begin by obtaining the asymptotics for the contributions of the
$I^{\pm}_s$-terms. As noted above it suffices to consider the integrals
\begin{equation}
  \frac{1}{|x'|^{\frac{n-3}{2}}}\int\limits_{r_0}^{\infty}J_{\frac{n-3}{2}}(r|x'|)
  r^{\frac{n-1}{2}} e^{-r|x_n|} I^{\pm}_s(r,x_n)dr\text{ for }0<\!<r_0.
\end{equation}
We split this analysis into the two cases: $|x'|/|x_n|\to 0$ and
$\infty.$

\begin{prop}\label{prop4.4.004}
  For $0<b$ and $0<s<\frac 12,$ there are functions $L^{\pm}_{s,b}(y)\in\cC^{\infty}([0,\infty)),$ so
  that, as $|x_n|+|x'|$ tends to zero, with $|x'|/|x_n|$ bounded above,
  \begin{equation}\label{eqn4.33.002}
  \frac{1}{|x'|^{\frac{n-3}{2}}}\int\limits_{0}^{\infty}J_{\frac{n-3}{2}}(r|x'|)
  r^{\frac{n-1}{2}} 
e^{-r|x_n|} I^{\pm}_s(r,x_n)dr
=\frac{L^{\pm}_{s,b}\left(\frac{|x'|}{|x_n|}\right)}
{(x_n^2+|x'|^2)^{\frac{n-2+2s}{2}}}+O\left(\frac{1}{(|x_n|+|x'|)^{n-3+4s}}\right).
  \end{equation}
\end{prop}

\begin{proof}
We choose an $r_0$ so that we can apply the remainder estimate
in~\eqref{eqn4.24.002}. If we denote the quantity on the left
of~\eqref{eqn4.33.002} by $\cI_{s}^{\pm}(|x'|,x_n),$ then it follows that, for
each $N$ there is a $C_N$ so that
\begin{multline}\label{eqn4.34.002}
  \cI_{s}^{\pm}(|x'|,x_n)=\sum_{j=0}^{N}\frac{a_j^{\pm}}{|x'|^{\frac{n-3}{2}}}\int\limits_{r_0}^{\infty}J_{\frac{n-3}{2}}(r|x'|)
  r^{\frac{n-1}{2}} e^{-r|x_n|} \frac{f_j(r|x_n|)dr}{r^{(1-2s)(1+j)}}+\\
\frac{C_N}{|x'|^{\frac{n-3}{2}}}\int\limits_{r_0}^{\infty}J_{\frac{n-3}{2}}(r|x'|)
  r^{\frac{n-1}{2}} e^{-r|x_n|} \frac{dr}{r^{(1-2s)(N+2)}}+O(1).
\end{multline}
We choose $N$ so that $n-2-(1-2s)(1+N)>-1$ and $n-2-(1-2s)(2+N)<-1.$ 
(Here we use our standing assumption that $s\notin\QQ.$  We leave
the equality case to the interested reader.) To estimate the remainder we use
the very crude estimate
\begin{equation}
  |J_{\frac{n-3}{2}}(r|x'|)|\leq C(r|x'|)^{\frac{n-3}{2}}.
\end{equation}
This estimate, along with our choice of $N$ shows that, without regard for the
behavior of the ratio $|x'|/|x_n|,$ the remainder term is $O(1)$
as $|x'|+|x_n|$ tends to zero.

To estimate the terms in the finite sum we change variables letting
$r|x_n|=\rho$ to obtain that
\begin{equation}
   \cI_{s}^{\pm}(|x'|,x_n)=\sum_{j=0}^{N}\frac{a_j^{\pm}|x_n|^{(1-2s)(1+j)}}{|x'|^{\frac{n-3}{2}}|x_n|^{\frac{n+1}{2}}}
\int\limits_{r_0|x_n|}^{\infty}J_{\frac{n-3}{2}}(\rho|x'|/|x_n|)
  \rho^{\frac{n-1}{2}} e^{-\rho} \frac{f_j(\rho)d\rho}{\rho^{(1-2s)(1+j)}}+O(1).
\end{equation}
Recalling that as $|x'|/|x_n|$ is bounded or tending to zero, these integrals
are easily evaluated by replacing the Bessel function with its power
series. Using the estimates for  $f_j$ that follow from
Lemmas~\eqref{fjsymbest} and~\eqref{lem4.1.002}, we obtain the formula
in~\eqref{eqn4.33.002}. Note that only the $j=0$ term contributes to the
leading order behavior.
\end{proof}

The analysis of the case with $|x'|/|x_n|$ tending to infinity is rather different. Nonetheless
we obtain the same leading order behavior in this limit as well.
  \begin{prop}
  For $0<b$ and $0<s<\frac 12,$ there are functions $K^{\pm}_{s,b}(y)\in\cC^{\infty}([0,\infty)),$ so
  that, as $|x_n|+|x'|$ tends to zero, with $|x_n|/|x'|$ bounded above,
  \begin{equation}\label{eqn4.38.002}
  \frac{1}{|x'|^{\frac{n-3}{2}}}\int\limits_{0}^{\infty}J_{\frac{n-3}{2}}(r|x'|)
  r^{\frac{n-1}{2}} 
e^{-r|x_n|} I^{\pm}_s(r,x_n)dr
=\frac{K^{\pm}_{s,b}\left(\frac{|x_n|}{|x'|}\right)}
{(x_n^2+|x'|^2)^{\frac{n-2+2s}{2}}}+O\left(\frac{1}{(|x_n|+|x'|)^{n-3+4s}}\right).
  \end{equation}
\end{prop}

\begin{proof}
  We begin, as before with the formula in~\eqref{eqn4.34.002}, where $N$ is
  chosen as before in the proof of Proposition~\ref{prop4.4.004}. The remainder
  term is still $O(1)$ as $|x'|+|x_n|$ tends to zero. To estimate the
  integrals, we return to the Fourier representation in the first line
  of~\eqref{eqn4.5.0002}, obtaining:
  \begin{equation}
     \cI_{s}^{\pm}(x_n,|x'|)=\sum_{j=0}^{N}a_j^{\pm}\int\limits_{r_0<|\eta'|}
e^{i|x'|\omega\cdot\eta'}e^{-|\eta'||x_n|}
     \frac{f_j(|\eta'||x_n|)d\eta'}{|\eta'|^{(1-2s)(1+j)}}+ O(1),
  \end{equation}
where $x'=|x'|\omega.$ We now let $|x_n|\eta'=\xi'$ to get
\begin{equation}
     \cI_{s}^{\pm}(x_n,|x'|)=\sum_{j=0}^{N}\frac{a_j^{\pm}|x_n|^{(1-2s)j}}{|x_n|^{n-2+2s}}
\int\limits_{r_0|x_n|<|\xi'|}
e^{i(|x'|/|x_n|)\omega\cdot\xi'}e^{-|\xi'|}
     \frac{f_j(|\xi'|)d\xi'}{|\xi'|^{(1-2s)(1+j)}}+ O(1),
  \end{equation}
  As follows from Lemma~\ref{fjsymbest}, the integrand behaves symbolically as
  $|\xi'|$ tends to infinity. Hence it follows easily that if
  $\varphi\in\cC^{\infty}(\RR^{n-1})$ is $1$ in a small neighborhood of zero
  and $0$ where $|\xi'|>\frac 12,$ then, for any $N\in\NN,$
\begin{multline}
     \cI_{s}^{\pm}(x_n,|x'|)=\\
\sum_{j=0}^{N}\frac{a_j^{\pm}|x_n|^{(1-2s)j}}{|x_n|^{n-2+2s}}
\left[\int\limits_{r_0|x_n|<|\xi'|}\varphi(|\xi'|^2)
e^{i(|x'|/|x_n|)\omega\cdot\xi'}e^{-|\xi'|}
     \frac{f_j(|\xi'|)d\xi'}{|\xi'|^{(1-2s)(1+j)}}+O\left(\left[\frac{|x_n|}{|x'|}\right]^N\right)\right]+ O(1).
  \end{multline}
As before, we make a bounded error replacing the integral over $\{\xi':\:r_0|x_n|<|\xi'|\}$
with the integral over $\RR^{n-1}.$ 

Thus, applying Lemma~\ref{lem4.1.002}, we are left to consider
\begin{multline}
     \cI_{s}^{\pm}(x_n,|x'|)=
\sum_{j=0}^{N}\frac{a_j^{\pm}|x_n|^{(1-2s)j}}{|x_n|^{n-2+2s}}\\
\left[\int\limits_{\RR^{n-1}}\varphi(|\xi'|^2)
e^{i(|x'|/|x_n|)\omega\cdot\xi'}e^{-|\xi'|}
     \frac{\left[f_{j1}(|\xi'|)+f_{j0}(|\xi'|)|\xi'|^{(1-2s)(j+1)}\right]d\xi'}{|\xi'|^{(1-2s)(1+j)}}+
O\left(\left[\frac{|x_n|}{|x'|}\right]^N\right)\right]+ O(1).
  \end{multline}
  Recall that the functions $\{f_{j0}, f_{j1}\}$ are smooth in $[0,1).$ As we
  are considering the limit $|x'|/|x_n|\to\infty,$ the assertion of the
proposition now follows from a standard stationary phase argument.
\end{proof}
This gives the leading order asymptotic behavior of $E_{s,b}(x',x_n)$ in the
  half plane $x_n<0.$ 

\subsection{Asymptotics of the Residue Term}

To complete the analysis in the other half plane we need
  to assess the contribution of the residue term, i.e. the behavior of the integral
  \begin{equation}
    \cH(|x'|,x_n)=\frac{c_n\omega_{n-2}}{|x'|^{\frac{n-3}{2}}}
\int\limits_{r_0}^{\infty}J_{\frac{n-3}{2}}(r|x'|) r^{\frac{n-1}{2}} H(r,x_n)dr,
  \end{equation}
as $|x'|+x_n\to 0^+.$ For this case we have the result:
\begin{prop}\label{prop4.5}
  Suppose that $0<b$ and $0<s<\frac 12,$ and $x_n>0.$ There is a function
  $L_{s,b}(u,v)
\in\cC^{0}([0,\infty]\times [0,\infty]),$ so
  that, as $|x_n|+|x'|$ tends to zero, 
  \begin{equation}\label{eqn4.44.0001}
  \frac{c_n}{|x'|^{\frac{n-3}{2}}}\int\limits_{0}^{\infty}J_{\frac{n-3}{2}}(r|x'|)
  r^{\frac{n-1}{2}} 
H(r,x_n)dr
=\frac{L_{s,b}\left(\frac{x_n^{\frac 1s}}{|x'|^2},x_n\right)}
{(x_n^{\frac 1s}+|x'|^2)^{\frac{n-1}{2}}}\left[1+O\left((x_n^{\frac{1}{s}}+|x'|^2)^{1-2s}\right)\right].
  \end{equation}
The function $L_{s,b}\left(\frac{x_n^{\frac 1s}}{|x'|^2},x_n\right)$ has a finite limit
as $\frac{x_n^{\frac 1s}}{|x'|^2}$ tends to infinity and behaves like
$\frac{x_n^{\frac 1s}}{|x'|^2}$ as this quantity tends to zero.
\end{prop}
\begin{proof}
We first consider the behavior as $\frac{x_n^{\frac 1s}}{|x'|^2}$ tends to
infinity. Recall from~\eqref{eqn4.11} that
\begin{equation}
  H(r,x_n)=2\pi
  e^{-rx_ny(\beta)}\left(\frac{1-y^2(\beta)}{b(1-(1-2s)y^2(\beta))}\right),\text{
    where }\beta=br^{1-2s}.
\end{equation}
 For large $\beta,$ we rewrite 
\begin{equation}
  y(\beta)=\frac{r^{2s-1}}{b}(1+\ty(\beta)),
\end{equation}
where $\ty(\beta)=\sum_{j=1}^{\infty}\frac{b_j}{\beta^{2j}},$ so that
\begin{multline}
    \cH(|x'|,x_n)=\frac{c_n\omega_{n-2}}{b|x'|^{\frac{n-3}{2}}}\int\limits_{r_0}^{\infty}J_{\frac{n-3}{2}}(r|x'|)
    r^{\frac{n-1}{2}} e^{-\frac{r^{2s}x_n}{b}(1+\ty(br^{1-2s}))}\times\\
\left(\frac{b^2-r^{2(2s-1)}(1+\ty(br^{1-2s}))^2}{b^2-(1-2s)r^{2(2s-1)}(1+\ty(br^{1-2s}))^2}\right)dr.
  \end{multline}
As noted above the contribution from $0$ to $r_0$ is uniformly $O(1).$

Inserting the Taylor series for $J_{\frac{n-3}{2}}$ and 
letting $x_n^{\frac{1}{2s}}r=w,$ we obtain 
\begin{multline}
    \cH(|x'|,x_n)=\frac{c_n\omega_{n-2}}{bx_n^{\frac{n-1}{2s}}}
\sum_{j=0}^{\infty}a_j\left(\frac{|x'|^2}{x_n^{\frac 1s}}\right)^j\int\limits_{x_n^{\frac{1}{2s}}r_0}^{\infty}
    w^{n-2+2j} e^{-\frac{w^{2s}}{b}(1+Y(w))}\times\\
\left(\frac{b^2w^{2(1-2s)}-x_n^{\frac{1-2s}{s}}(1+Y(w))^2}{b^2w^{2(1-2s)}-(1-2s)x_n^{\frac{1-2s}{s}}(1+Y(w))^2}\right)dw,
  \end{multline}
where $Y(w)=\ty(b(w/x_n^{\frac{1}{2s}})^{1-2s}).$ From the behavior of
$y(\beta)$ it is not difficult to show that, so long as
$x_n^{\frac{1}{2s}}/|x'|$ is bounded away from zero, this takes the form
\begin{equation}
  \cH(|x'|,x_n)=\frac{h\left(\frac{|x'|^2}{x_n^{\frac{1}{s}}},x_n\right)} 
{(x_n^{\frac{1}{s}}+|x'|^2)^{\frac{n-1}{2}}},
\end{equation}
where $h(u,v)$ is a continuous function in $[0,\infty)\times [0,\infty),$ with
$h(u,0)$ equal to
\begin{equation}
   h(u,0)=
\sum_{j=0}^{\infty}a'_ju^j\int\limits_{0}^{\infty}
    w^{n-2+2j} e^{-w^{2s}}dw.
  \end{equation}

All that remains is to analyze the behavior of $\cH(|x'|,x_n)$ as
$x_n^{\frac{1}{2s}}/|x'|$ tends to zero. To that end we proceed as before,
rewriting $\cH$ as
\begin{equation}
    \cH(|x'|,x_n)=
\int\limits_{r_0<|\eta'|} e^{i|x'|\omega\cdot\eta'} H(|\eta'|,x_n)d\eta'.
  \end{equation}
We now let $x_n^{\frac{1}{2s}}\eta'=\xi'$ to obtain that 
\begin{equation}
  \cH(|x'|,x_n)=
\frac{1}{x_n^{\frac{n-1}{2s}}}
\int\limits_{x_n^{\frac{1}{2s}}r_0<|\xi'|} e^{i(|x'|/x_n^{\frac{1}{2s}})\omega\cdot\xi'} H(|\xi'|/x_n^{\frac{1}{2s}},x_n)d\xi' 
\end{equation} 
The analysis in this case is quite similar to that employed above. The
integrand involves the function $H(|\xi'|/x_n^{\frac{1}{2s}},x_n)$ where the
first argument is large. We therefore make use of~\eqref{eqn4.15.004} to obtain
the convergent expansion:
\begin{equation}
  \left(\frac{1-y^2(\beta)}{b(1-(1-2s)y^2(\beta))}\right)=\frac{1}{b}\left(1+\sum_{j=1}^{\infty}
\frac{\tb_j}{\beta^{2j}}\right).
\end{equation}
From which it follows easily that
\begin{equation}\label{eqn4.5.01}
  H\left(\frac{|\xi'|}{x_n^{\frac{1}{2s}}},x_n\right)=\frac{2\pi e^{-\frac{|\xi'|^{2s}}{b}(1+\ty(\tbeta))}}{b}
\left[1+\sum_{j=1}^{\infty}\frac{\tb_j}{\tbeta^{2j}}\right],
\end{equation}
where
\begin{equation}
  \tbeta=b\left(\frac{|\xi'|}{x_{n}^{\frac{1}{2s}}}\right)^{1-2s}.
\end{equation}
The principal contribution comes from
\begin{equation}
  \frac{2\pi
  }{bx_n^{\frac{n-1}{2s}}}
\int\limits_{x_n^{\frac{1}{2s}}r_0<|\xi'|}
e^{i\left[\frac{|x'|\omega\cdot\xi'}{x_n^{\frac{1}{2s}}}\right]}e^{-\frac{|\xi'|^{2s}}{b}(1+\ty(\tbeta))}d\xi'.
\end{equation}

From the symbolic properties of $\ty(\tbeta)$ it is not difficult to see that 
inserting a compactly supported function of the form $\varphi(|\xi'|^2),$ with
$\varphi(r)=1$ in a neighborhood of the origin, leads to an error that is
$$O\left(|x'|^{1-n}\left[\frac{x_n^{\frac{1}{2s}}}{|x'|}\right]^N\right),$$ 
for any $N,$ and therefore, up to error terms of this order, we set
 \begin{equation}
  \cH_{\pr}(|x'|,x_n)= \frac{2\pi
  }{bx_n^{\frac{n-1}{2s}}}
\int\limits_{x_n^{\frac{1}{2s}}r_0<|\xi'|} \varphi(|\xi'|^2)
e^{i\left[\frac{|x'|\omega\cdot\xi'}{x_n^{\frac{1}{2s}}}\right]}e^{-\frac{|\xi'|^{2s}}{b}(1+\ty(\tbeta))}d\xi'.
\end{equation}
Using a Taylor series we estimate the exponential to get 
\begin{equation}\label{eqn4.56.001}
  \cH_{\pr}(|x'|,x_n)= \frac{2\pi}{bx_n^{\frac{n-1}{2s}}}
\int\limits_{x_n^{\frac{1}{2s}}r_0<|\xi'|} \varphi(|\xi'|^2)
e^{i\left[\frac{|x'|\omega\cdot\xi'}{x_n^{\frac{1}{2s}}}\right]}\left[1-\frac{|\xi'|^{2s}}{b}+
O\left(|\xi'|^{4s}+|\xi'|^{2s}\left(\frac{x_n^{\frac{1}{2s}}}{|\xi'|}\right)^{2(1-2s)}\right)\right]d\xi'.
\end{equation}
Replacing this last expression with the integral over $\RR^{n-1}$ again
introduces a bounded error, thus
\begin{equation}
  \cH_{\pr}(|x'|,x_n)= -\frac{2\pi}{bx_n^{\frac{n-1}{2s}}}
\int\limits_{\RR^{n-1}} \varphi(|\xi'|^2)
e^{i\left[\frac{|x'|\omega\cdot\xi'}{x_n^{\frac{1}{2s}}}\right]}\left[\frac{|\xi'|^{2s}}{b}+
O\left(|\xi'|^{4s}+|\xi'|^{2s}\left(\frac{x_n^{\frac{1}{2s}}}{|\xi'|}\right)^{2(1-2s)}\right)\right]d\xi'.
\end{equation}
Here we have also taken account of the fact that the ``$1$'' term
in~\eqref{eqn4.56.001} contributes a term of order
\begin{equation}
  \frac{1}{|x'|^{n-1}}O\left(\left[\frac{x_n}{|x'|^{2s}}\right]^M\right),
\end{equation}
for any $M.$ Using a standard stationary phase
argument we see that
\begin{equation}
  \cH_{\pr}(|x'|,x_n)=
  \frac{c_{n,s}}{b|x'|^{n-1}}\left[\frac{x_n}{|x'|^{2s}}\right]\left(1+
O\left(\left[\frac{x_n}{|x'|^{2s}}\right]^2
+|x'|^{2(1-2s)}\right)\right).
\end{equation}

An analysis similar to that used above in~\eqref{eqn4.34.002} is used to
analyze the sum over $\tbeta^{-2j}$ in~\eqref{eqn4.5.01}. As before we divide
the sum in~\eqref{eqn4.5.01} into two parts. In the first part $n-2j(1-2s)>1,$
and in the later it is $n-2j(1-2s)<1.$ The contribution of second part remains
bounded as $|x'|$ and $x_n$ tend to zero and so can be ignored. Arguing as in
the estimation of $\cH_{\pr}(|x'|,x_n),$ we show that the difference
$\cH(|x'|,x_n)- \cH_{\pr}(|x'|,x_n)$ is estimated by terms of the type in the
$O$-term above, and therefore
\begin{equation}\label{eqn4.59.001}
  \cH(|x'|,x_n)=
  \frac{c_{n,s}}{b(x_n^{\frac 1s}+|x'|^{2})^{\frac{n-1}{2}}}\left[\frac{x_n}{|x'|^{2s}}\right]\left(1+
O\left(\left[\frac{x_n}{|x'|^{2s}}\right]+|x'|^{2(1-2s)}\right)\right).
\end{equation}
\end{proof}

\begin{rmk}
The leading order singularity  in the upper half space is $
\frac{c_{n,s}}{b(x_n^{\frac 1s}+|x'|^{2})^{\frac{n-1}{2}}};$  in the lower
half space it is the weaker singularity
$\frac{c'_{n,s}}{b(x_n^2+|x'|^{2})^{\frac{n-2+2s}{2}}}$ that dominates. The
formula in~\eqref{eqn4.59.001} shows that the coefficient of the stronger
singularity vanishes as $\frac{x_n}{|x'|^{2s}}$ tends to zero. Notice also that
  the homogeneities of the leading singularities differ in the half planes. 
\end{rmk}

\appendix

\section{The analysis of $p^0_s(x';x_n)$ }
The fact that the integrand in~\eqref{eqn4.11} has
only a single pole in the upper half planes shows that the behavior of
the $\tau$-integral is quite different for $x_n>0$ and $x_n<0.$  As noted this echoes the
behavior of the kernel, $p^0_s(x';x_n),$ for the heat operator
$e^{-x_n(-\Delta_{\RR^{n-1}})^s},$ which vanishes identically in the set
$\{(x',x_n):\: x_n<0\}.$ In this appendix we give a precise description of this
heat kernel. 

An elementary calculation shows that where $x_n>0$ we have:
\begin{equation}\label{eqnA1.004}
\begin{split}
  p^0_s(x';x_n)&=\frac{2\pi
    c_n\omega_{n-2}}{|x'|^{\frac{n-3}{2}}}\int\limits_{0}^{\infty}
J_{\frac{n-3}{2}}(r|x'|) r^{\frac{n-1}{2}}e^{-x_nr^{2s}}dr\\
&=
\frac{2\pi
    c_n\omega_{n-2}}{|x'|^{n-1}}\int\limits_{0}^{\infty}
J_{\frac{n-3}{2}}(r) r^{\frac{n-1}{2}}e^{-\frac{x_n}{|x'|^{2s}}r^{2s}}dr.
\end{split}
\end{equation}
This kernel is an-isotropically homogeneous:
\begin{equation}
  p^0_s(\mu x';\mu^{2s}x_n)=\mu^{-(n-1)}p^0_s(x';x_n).
\end{equation}
When $x_n/|x'|^{2s}$ is bounded away from zero, we can simply expand the Bessel
function in a power series to obtain that:
\begin{equation}\label{eqn4.15.001}
  p^0_s(x';x_n)=\frac{1}{|x_n|^{\frac{n-1}{2s}}}\sum_{j=0}^{\infty}b_j\left(\frac{|x'|^2}{x_n^{\frac{1}{s}}}\right)^{j},
\end{equation}
which is a smooth function of $|x'|^2/x_n^{\frac{1}{s}},$ where $x_n>0.$

A priori, it appears much subtler to deduce the behavior of $p^0_s(x';x_n)$ as
$x_n/|x'|^{2s}$ tends to zero.  We could use an argument similar to that used
in the proof of Proposition~\ref{prop4.5}, but there is an alternate
approach, using the notion of subordination, which we have decided to
employ. Using it we derive a second formula for $p^0_s$ from which these
asymptotics are more transparent. Indeed, this second formula gives a method to
compute the asymptotics of the integrals
\begin{equation}
  \int\limits_{0}^{\infty}
J_{\frac{n-3}{2}}(r) r^{\frac{n-1}{2}}e^{-\lambda r^{2s}}dr
\end{equation}
as $\lambda\to 0^+,$ for $0<s<1.$ These asymptotics are, by no means obvious,
as they result from global cancellations taking place within the integral. As
this analysis does not appear to be in the literature we pause to consider this
question. Our result is:
\begin{prop} For  $\nu\geq 0$ and $0<s<1,$ there exists a
  sequence $\{\mu_{sk}\}$ so that
  \begin{equation}\label{eqn4.12.001}
    \int\limits_{0}^{\infty}
J_{\nu}(r) r^{\nu+1}e^{-\lambda r^{2s}}dr\sim \sum_{k=1}^{\infty}\mu_{sk}\lambda^{k},
  \end{equation}
as $\lambda\to 0^+.$
\end{prop}
\begin{proof}
Let $\Phi_s(x)$ be the function defined by the Laplace transform formula:
\begin{equation}\label{eqn4.13.001}
  e^{-\lambda^s}=\int\limits_{0}^{\infty}\Phi_s(x)e^{-\lambda x}dx.
\end{equation}
The Laplace inversion formula shows that, for any $\sigma\in [0,\infty),$ we have
\begin{equation}\label{eqn4.16.003}
  \Phi_s(x)=\frac{1}{2\pi i}\int\limits_{\sigma-i\infty}^{\sigma+i\infty}e^{xz-z^s}dz,
\end{equation}
where $z^s,$ defined on $\CC\setminus (-\infty,0],$ is normalized to be a
positive real number for $z\in (0,\infty).$ 

If $L$ is a positive operator generating a
positivity preserving semigroup and $0<s<1,$ then
\begin{equation}
  e^{-tL^s}=\int\limits_{0}^{\infty}\frac{e^{-x L}}{t^{\frac 1s}}\Phi_s\left(\frac{x}{t^{\frac 1s}}\right)dx,
\end{equation}
see~\cite[\S IX.11]{Yosida}. Using~\eqref{eqn4.13.001} and~\cite[Formula 6.631.4]{GrRhy}, 
\begin{equation}
  \int\limits_{0}^{\infty}x^{\nu +1}J_{\nu}(x)e^{-\alpha
    x^2}dx=\frac{1}{(2\alpha)^{\nu+1}}
\exp\left(\frac{-1}{4\alpha}\right),
\end{equation}
we easily prove:
\begin{lem}
  For $0<s<1$ and $0\leq \nu,$ there is a constant $C_{\nu}$ so that the identity
  \begin{equation}\label{eqn4.16.001}
    \int\limits_{0}^{\infty}\frac{\Phi_s(\sigma)e^{-\frac{1}{4\sigma\lambda^{\frac
            1s}}}d\sigma}{(\lambda^{\frac 1s}\sigma)^{\nu+1}}=
C_{\nu}\int\limits_{0}^{\infty}J_{\nu}(\tau)e^{-\lambda\tau^{2s}}\tau^{\nu+1}d\tau
  \end{equation}
holds for $\lambda$ with $\Re\lambda>0.$
\end{lem}

From formula~\eqref{eqn4.16.003} it is clear that
$\Phi_s(x)$ is in $\cC^{\infty}(\RR)$ and vanishes where $x<0.$ For real,
positive $x,$ we can replace this formula with
\begin{equation}
   \Phi_s(x)=\Re\left[\frac{1}{\pi
       i}\int\limits_{0}^{\infty}e^{xe^{i\theta}r-(e^{i\theta}r)^s}e^{i\theta} dr\right]
\end{equation}
for $\theta\in [\frac{\pi}{2},\pi].$ Taking $\theta=\pi,$ and using the
Taylor series for $e^{-e^{is\pi}r^s}$ gives,  for $x>0,$ that
\begin{equation}
  \Phi_s(x)=\Re\left[\frac{i}{\pi}
       \int\limits_{0}^{\infty}e^{-xr-e^{is\pi}r^s}dr\right]
			=\frac{1}{\pi}\sum_{j=1}^{\infty}
\frac{(-1)^{j-1}\Gamma(sj+1)\sin\pi s j}{j!}\frac{1}{x^{sj+1}}.
\end{equation}
The asymptotic expansion in~\eqref{eqn4.12.001} follows
from~\eqref{eqn4.16.001} and this expansion.
\end{proof}
\begin{rmk} The identity in~\eqref{eqn4.16.001} trades the $\lambda$ in the
  exponent on the right hand side for $1/\lambda^{\frac 1s}$ on the left hand
  side. This in turn allows us to trade the very subtle, globally determined
  asymptotic behavior of the integral on the right hand side as $\lambda\to
  0^+,$ for the much simpler asymptotic analysis, as $1/\lambda^{\frac
    1s}\to\infty$ on the left. Paying somewhat closer attention to the
  coefficients in the expansion shows that if $2s<1,$ then the infinite sum on
  the right hand side of~\eqref{eqn4.12.001}, given by,
  \begin{equation}
    \frac{1}{\pi}\sum_{j=1}^{\infty}\frac{\Gamma(sj+1)\Gamma(sj+\nu+1)\sin(s\pi j)}{j!}(-1)^{j-1}(4^s\lambda)^{j},
  \end{equation}
 is an entire function. If $s=1/2,$ then this series has a finite radius of
 convergence, and if $1/2<s<1,$ then it is an asymptotic series.
\end{rmk}

Using this result we see from~\eqref{eqnA1.004} and~\eqref{eqn4.12.001} that,
as $x_n/|x'|^{2s}\to 0,$ we have:
\begin{equation}
  p^0_s(x';x_n)\sim \frac{\chi_{[0,\infty)}(x_n)}
{|x'|^{n-1}}\left(\frac{c_n x_n}{|x'|^{2s}}+O\left[\frac{x_n}{|x'|^{2s}}\right]^2\right).
\end{equation}
The asymptotic result can also be obtained using standard microlocal
techniques.  Note, however, that our formula is valid whether or not both $|x'|$
and $x_n$ tend to zero, and shows that away from $x'=0$ this kernel vanishes to
order 1 as $x_n\to 0^+.$ It also shows that for fixed $x_n>0,$ this kernel
behaves like $1/|x'|^{n-1+2s}$ as $|x'|$ tends to infinity. Combining this
formula with~\eqref{eqn4.15.001} establishes the following result.
\begin{prop}
  The kernel $p^0_s(x';x_n)$  for the operator
  $e^{-x_n(-\Delta_{\RR^{n-1}})^s},$ has the following representation
  \begin{equation}
p^0_s(x';x_n)=
\chi_{[0,\infty)}(x_n)\frac{H^+\left(\frac{|x'|^2}{x_n^{\frac{1}{s}}}\right)}
{(x_n^{\frac{1}{s}}+|x'|^2)^{\frac{n-1}{2}}},
  \end{equation}
where $H^+(\lambda)$ is a smooth positive function in $[0,\infty)$ with an
expansion of the form
\begin{equation}
  H^+(\lambda)=\frac{c_n}{\lambda^s}\left( 1+O(1/\lambda^s)\right),
\end{equation}
as $\lambda\to\infty.$
\end{prop}

%
%
\def\cprime{$'$} \def\polhk#1{\setbox0=\hbox{#1}{\ooalign{\hidewidth
  \lower1.5ex\hbox{`}\hidewidth\crcr\unhbox0}}} \def\cprime{$'$}
  \def\cprime{$'$} \def\cprime{$'$}
  \def\lfhook#1{\setbox0=\hbox{#1}{\ooalign{\hidewidth
  \lower1.5ex\hbox{'}\hidewidth\crcr\unhbox0}}} \def\cprime{$'$}
  \def\cprime{$'$} \def\cprime{$'$} \def\cprime{$'$} \def\cprime{$'$}
\providecommand{\bysame}{\leavevmode\hbox to3em{\hrulefill}\thinspace}
\providecommand{\MR}{\relax\ifhmode\unskip\space\fi MR }
\providecommand{\MRhref}[2]{%
  \href{http://www.ams.org/mathscinet-getitem?mr=#1}{#2}
}
\providecommand{\href}[2]{#2}


\end{document}